\documentclass[reqno]{amsart}
\usepackage{amsthm,amsmath,amssymb}
\usepackage{stmaryrd,mathrsfs}
\usepackage[T1]{fontenc}
\usepackage{esint}
\usepackage{tikz}

\allowdisplaybreaks

\newcommand{\ud}[0]{\,\mathrm{d}}

\newcommand{\abs}[1]{|#1|}
\newcommand{\Babs}[1]{\Big|#1\Big|}

\newcommand{\Norm}[2]{\|#1\|_{#2}}

\newcommand{\BNorm}[2]{\Big\|#1\Big\|_{#2}}
\newcommand{\pair}[2]{\langle #1,#2 \rangle}
\newcommand{\Bpair}[2]{\Big\langle #1,#2 \Big\rangle}
\newcommand{\ave}[1]{\langle #1\rangle}
\newcommand{\cave}[1]{\langle\!\langle #1\rangle\!\rangle}
\newcommand{\sset}[1]{\langle\!\langle #1\rangle\!\rangle}

\newcommand{\lspan}[0]{\operatorname{span}}

\newcommand{\bddlin}[0]{\mathscr{L}}

\newcommand{\BMO}[0]{\operatorname{BMO}}
\newcommand{\supp}[0]{\operatorname{supp}}
\newcommand{\loc}[0]{\operatorname{loc}}

\newcommand{\R}{\mathbb{R}}

\newcommand{\N}{\mathbb{N}}
\newcommand{\Z}{\mathbb{Z}}

\newcommand{\eps}[0]{\varepsilon}

\newcommand{\avL}[0]{\textit{\L}}

\swapnumbers \numberwithin{equation}{section}

\theoremstyle{plain}
\newtheorem{theorem}[equation]{Theorem}
\newtheorem{proposition}[equation]{Proposition}
\newtheorem{corollary}[equation]{Corollary}
\newtheorem{lemma}[equation]{Lemma}

\theoremstyle{definition}
\newtheorem{definition}[equation]{Definition}

\theoremstyle{remark}
\newtheorem{remark}[equation]{Remark}
\newtheorem{example}[equation]{Example}

\makeatletter
\@namedef{subjclassname@2010}{%
  \textup{2010} Mathematics Subject Classification}
\makeatother

\begin{document}

\title{Some remarks on convex body domination}

\author[T.~P.\ Hyt\"onen]{Tuomas P.\ Hyt\"onen}
\address{Department of Mathematics and Statistics, P.O.B.~68 (Pietari Kalmin katu~5), FI-00014 University of Helsinki, Finland}
\email{tuomas.hytonen@helsinki.fi}

\date{\today}

\thanks{The author is supported the Academy of Finland via the Finnish Centre of Excellence in Randomness and Structures ``FiRST'' (grant no.\ 346314).}



\maketitle

\begin{center}
Dedicated, with admiration, to the Ukrainian people.
\end{center}

\begin{abstract}
Convex body domination is an important elaboration of the technique of sparse domination that has seen significant development and applications over the past ten years. In this paper, we present an abstract framework for convex body domination, which also applies to Banach space -valued functions, and yields matrix-weighted norm inequalities in this setting. We explore applications to ``generalised commutators'', obtaining new examples of bounded operators among linear combinations of compositions of the form $a_iTb_i$, where $a_i,b_i$ are pointwise multipliers and $T$ is a singular integral operator.
\end{abstract}

\section{Introduction}

The technique of sparse domination was developed to provide a simpler approach, achieved by Lerner \cite{Lerner:simple}, to the ``$A_2$ conjecture'' on sharp weighted norm inequalities for Calder\'on--Zygmund operators, which was first proved with a different machinery by the author \cite{Hytonen:A2}. However, beyond this original aim, sparse domination immediately led to significant further consequences and has by now been applied to a variety of new questions, of which \cite{BHRT,BRS,BFP,CCDO:17,CKL}  is only a  sample. The method consists of two main steps that are largely independent of each other and essentially decouple the operator from the space or norm in which it should be estimated:
\begin{enumerate}
  \item\label{it:domOp} Dominating an operator of interest by a suitable sparse operator/form.
  \item\label{it:domSp} Estimating the sparse form with respect to relevant norms of interest.
\end{enumerate}

While sparse domination very efficiently captures the local size of an object under consideration, and this is precisely what is needed in many applications, it loses information about directions, which is sometimes relevant when dealing with vector-valued functions, and especially so, matrix-valued weights are involved. To extend the method to such questions, Nazarov et al. \cite{NPTV:convex} developed the so-called convex body domination, where the numerical averages featuring in sparse domination are replaced by convex subsets of $\R^n$, thus containing information about different behaviour in different directions. Since its introduction in the context of Calder\'on--Zygmund operators and matrix $A_2$ weights by \cite{NPTV:convex} (see also \cite{CDPO:unif} for another approach but based on the same key idea), convex body domination has been applied to matrix $A_p$-weight and two-weight bounds by Cruz-Uribe et al.~\cite{CUIM:18}, and extended to commutators of Calder\'on--Zygmund operators by Isralowitz et al.~\cite{IPR:21,IPT:commu} and rough singular integral operators by Di Plinio et al.~\cite{DPHL} and Muller and Rivera-R\'{\i}os \cite{MRR:22}. In a recent breakthrough, Bownik and Cruz-Uribe \cite{BCU} extended the Rubio de Francia algorithm, and its key application to weighted extrapolation, to matrix-valued weights, by further development of the convex body philosophy.

The aim of this paper is to further explore this technique, providing extensions, new applications and---hopefully---some additional insight into the abstract underlying mechanisms. We begin by developing a somewhat general framework, but our claims for originality in this regard are relatively mild, as most of the ideas are at least implicit in the previous works in the existing literature. A certain justification for this framework comes from the observation that it applies almost verbatim to the case of Banach space -valued functions. To be precise, given a Banach space $E$, we consider functions taking values in $E^n$, and develop a version of convex body domination applicable to weighted norm inequalities involving matrix weights $W:\R^d\to\R^{n\times n}$, acting on $E^n$ in the natural way. That is, we make no attempt towards a fully operator-valued theory of weighted norm inequalities in infinite dimensions, yet the results that we obtain are still new even in this more modest generality. In particular, if $E$ is a Banach space with the UMD property, the classical Hilbert transform extends boundedly to the matrix-weighted space $L^2(W;E^n)$ of $E^n$-valued functions; see Corollary \ref{cor:L2WE} for the result, and Section \ref{sec:L2WE} for the relevant definitions and background. A key to this extension is the observation that the convex bodies arising from our framework are still $\R^n$-valued in this generality---and not, for instance, $E^n$-valued, as one might have (and this author certainly had) initially expected. Thus the powerful Euclidean machinery, most notably the John ellipsoid theorem, is still available in this setting.

As for new applications of the theory, we build on a recent observation from Isralowitz et al. \cite{IPR:21,IPT:commu} that convex body domination of an operator $T$ bootstraps to a domination of its commutators $[b,T]=bT-Tb$ with pointwise multipliers. As we will explore in Section \ref{sec:commu}, this phenomenon is far more general, and can be used to estimate any operators of the form
\begin{equation*}
  f\mapsto \sum_{i=1}^n a_i T(b_i f),
\end{equation*}
where an operator $T$ satisfying convex body domination is pre- and post-composed with pointwise multipliers $a_i,b_i$. From this general principle, we can in particular recover and sharpen a recent sufficient condition \cite{HLO:20} for the boundedness of iterated mixed commutators $[b^1,[b^2,T]]$ in terms of joint conditions on the pair of functions $(b^1,b^2)$, but also obtain new examples.

In contrast to the development of the abstract framework in the first part of the paper, we have not strived for the greatest generality in terms of the applications in the later sections. In many cases, it will be clear to an experienced reader that several variants and extensions could be obtained, and some of them will most likely be pursued in forthcoming works, by this author and others. Besides the concrete results contained in this paper, our aim is to hint at the many rich directions for the further development of the theory.

\section{Norms and convex bodies}

Let $X$ be a real normed space. We denote by
\begin{equation*}
  \bar B_X:=\{x\in X:\Norm{x}{X}\leq 1\}
\end{equation*}
its closed unit ball, and by $X^*$ the normed dual, which is a Banach space.
For $x^*\in X^*$, we define, as usual,
\begin{equation*}
  \Norm{x^*}{X^*}:=\sup\{\abs{\pair{x}{x^*}}:x\in\bar B_X\}.
\end{equation*}
As a consequence of the Hahn--Banach theorem, we have
\begin{equation}\label{eq:HB}
  \Norm{x}{X}=\sup\{\abs{\pair{x}{x^*}}:x^*\in\bar B_{X^*}\}
  =\max\{\pair{x}{x^*}:x^*\in\bar B_{X^*}\};
\end{equation}
in particular, the supremum is reached as a maximum, and we have $\Norm{x}{X}=\pair{x}{x^*}$ for some $x^*\in\bar B_{X^*}$.

For $\vec x=(x_i)_{i=1}^n\in X^n$ and $x^*\in X^*$, we define the $\R^n$-valued pairing $\pair{\vec x}{x^*}:=(\pair{x_i}{x^*})_{i=1}^n\in\R^n$ and the set-valued ``norm''
\begin{equation*}
  \sset{\vec x}_X:=\{\pair{\vec x}{x^*}:x^*\in \bar B_{X^*}\}\subset\R^n.
\end{equation*}

\begin{remark}
The notation is adapted from Nazarov et al.~\cite{NPTV:convex}, who introduced the version with $X=\avL^1(Q)$, the space $L^1(Q)$ with the normalised norm $\frac{1}{\abs{Q}}\Norm{\ }{1})$. The extension to $X=\avL^p(Q)$ (i.e., $L^p(Q)$ with the normalised norm $\frac{1}{\abs{Q}^{1/p}}\Norm{\ }{p})$ is due to Di Plinio et al.~\cite{DPHL}. Although our main applications will be concerned with spaces of functions (living on a cube $Q$), we find it illuminating to develop the basics of the theory on a completely abstract level. Among other things, this point of view will make it clear that there will be essentially no difference in treating a space $X=L^p(Q;E)$ of $E$-valued functions for an arbitrary Banach space $E$; for $\vec{f}\in X^n$, the corresponding $\cave{\vec f}_{X}$ will still be subsets of $\R^n$ and not, say, of $E^n$. This will allow us to make effortless use of the powerful John ellipsoid theorem from Euclidean geometry, even when working with functions taking values in an infinite-dimensional Banach space! In other applications, a choice like $X=L\log L(Q)$ might also be relevant.
\end{remark}

For $\vec a\in\R^n$ and $\vec x\in X^n$, we define the $X$-valued dot product
\begin{equation*}
 \vec a\cdot\vec x:=\vec x\cdot\vec a:=\sum_{i=1}^n a_i x_i. 
\end{equation*}
We observe the easy identities
\begin{equation*}
  \vec a\cdot\pair{\vec x}{x^*}=\pair{\vec a\cdot\vec x}{x^*},\quad\forall\vec a\in\R^n,\ \vec x\in X^n,\ x^*\in X^*,
\end{equation*}
and
\begin{equation*}
  \lspan_X(\vec{x}):=\lspan\{x_i\}_{i=1}^n=\{\vec a\cdot\vec x:a\in\R^n\}\subset X.
\end{equation*}

\begin{lemma}
For each $\vec x\in X^n$, the set $\sset{\vec x}_X\subset\R^n$ is convex, compact, and symmetric about the origin.
\end{lemma}

\begin{proof}
Symmetry, convexity and boundedness are immediate from the fact that $\bar B_{X^*}$ has these properties. For compactness in $\R^n$, it remains to show closedness, so suppose that $\pair{\vec x}{x_k^*}\to\vec e\in\R^n$ as $k\to\infty$, where each $x_k^*\in\bar B_{X^*}$; we need to show that $\vec{e}\in\sset{x}_X$. For each $\vec a\in\R^n$, it follows that
\begin{equation*}
  \abs{\vec a\cdot\vec e}=\lim_{k\to\infty}\abs{\vec{a}\cdot\pair{\vec x}{x_k^*}}
  =\lim_{k\to\infty}\abs{\pair{\vec{a}\cdot\vec x}{x_k^*}}
  \leq\Norm{\vec{a}\cdot\vec x}{X}.
\end{equation*}
This in turn implies that
\begin{equation*}
  \Lambda(\vec a\cdot\vec x):=\vec a\cdot\vec e,\quad \forall\vec a\cdot\vec x\in\lspan_X(\vec{x}),
\end{equation*}
gives a well-defined linear functional of norm $1$ on the subspace $\lspan_X(\vec{x})\subset X$. By the Hahn--Banach theorem, $\Lambda$ is the restriction of some $x^*\in\bar B_{X^*}$. Hence
\begin{equation*}
  \vec{a}\cdot\vec{e}=\Lambda(\vec a\cdot\vec x)=\pair{\vec{a}\cdot\vec{x}}{x^*}=\vec{a}\cdot\pair{\vec{x}}{x^*}\quad\forall\vec a\in\R^n,
\end{equation*}
and thus $\lim_{k\to\infty}\pair{\vec x}{x_k}=\vec{e}=\pair{\vec{x}}{x^*}\in\sset{\vec x}_X$, as we wanted to show.
\end{proof}

For $A,B\subset\R^n$, we define the Minkowski dot product
\begin{equation*}
  A\cdot B:=\{\vec a\cdot\vec b:\vec a\in A,\vec b\in B\}\subset\R.
\end{equation*}
If $A,B\subset\R^n$ are convex, compact and symmetric, so is $A\cdot B\subset\R$. On $\R$, such sets are precisely intervals of the form $[-c,c]$. Hence we can, and sometimes will, identify $A\cdot B=[-c,c]\subset\R$ with its right end-point $c\in[0,\infty)$.
In particular, for $\vec x\in X^n$ and $\vec y\in Y^n$, we will use this identification when dealing with
\begin{equation*}
\begin{split}
  \sset{\vec x}_X \cdot\sset{\vec y}_{Y}
  &=\{\pair{\vec x}{x^*}\cdot\pair{\vec y}{y^*}:x^*\in\bar B_{X^*},y^*\in\bar B_{Y^*}\} \\
  &=\Big\{\sum_{i=1}^n\pair{x_i}{x^*}\pair{y_i}{y^*}:x^*\in\bar B_{X^*},y^*\in\bar B_{Y^*}\Big\}.
\end{split}
\end{equation*}

\section{Bi-linear forms}

Let $X,Y$ be real normed spaces, and suppose that we have a bilinear from $t:X\times Y\to\R$. We define its extension acting on pairs of vectors $(\vec x,\vec y)\in X^n\times Y^n$ as follows. If $\vec e\in\R^n$ and $x\in X^n$, we have $\vec x\cdot\vec e\in F$ by our previous convention about the $X$-valued dot product. If $(e_i)_{i=1}^n$ is a fixed orthonormal basis of $\R^n$, we then define
\begin{equation*}
  t(\vec f,\vec g):=\sum_{i=1}^n t(\vec f\cdot\vec e_i,\vec g\cdot\vec e_i),
\end{equation*}
For $x\in X$, $y\in Y$, and $\vec e,\vec u\in\R^n$, it follows that
\begin{equation*}
  t(x \vec e,y \vec u):=\sum_{i=1}^n t(x \vec e\cdot\vec e_i,y \vec u\cdot\vec e_i)
  =t(x,y) \sum_{i=1}^n (\vec e\cdot\vec e_i)(\vec u\cdot\vec e_i)
  =t(x,y)\vec e\cdot\vec u.
\end{equation*}
If $(\vec u_i)_{i=1}^n$ is another orthonormal basis, then
\begin{equation*}
\begin{split}
  \sum_{i=1}^n t(\vec x\cdot\vec  u_i,\vec y\cdot\vec  u_i)
  &=\sum_{i,j,k=1}^n t(\vec x\cdot\vec  e_j,\vec y\cdot\vec  e_k)(\vec e_j\cdot \vec u_i)(\vec e_k\cdot\vec  u_i) \\
  &=\sum_{j,k=1}^n t(\vec x\cdot \vec e_j,\vec y\cdot \vec e_k)(\vec e_j\cdot \vec e_k) 
  =\sum_{j=1}^n t(\vec x\cdot \vec e_j,\vec y\cdot \vec e_j)=:t(\vec x,\vec y),
\end{split}
\end{equation*}
so the definition of $t(\vec f,\vec g)$ is independent of the chosen orthonormal basis.

If $A\in\R^{n\times n}$ is a linear transformation of $\R^n$, acting in a natural way on $F^n$, then
\begin{equation*}
\begin{split}
  t(A\vec x,\vec y)
  &=\sum_{i=1}^n t(A\vec f\cdot\vec e_i,\vec g\cdot\vec e_i)
  =\sum_{i=1}^n t(\vec f\cdot A^t\vec e_i,\vec g\cdot\vec e_i) \\
  &=\sum_{i,j=1}^n t(\vec f\cdot \vec e_j,\vec g\cdot\vec e_i)(\vec e_j\cdot A^t\vec e_i) \\
  &=\sum_{i,j=1}^n t(\vec f\cdot \vec e_j,\vec g\cdot\vec e_i)(A\vec e_j\cdot \vec e_i) 
  =\sum_{j=1}^n t(\vec f\cdot \vec e_j,\vec g\cdot A\vec e_j)
  =t(\vec f,A^t\vec g).
\end{split}
\end{equation*}

\section{From norm bounds to convex body bounds}

The idea of the following lemma lies behind many of the existing convex body domination results. To isolate the key point, we state it here in an operator-free version, involving functions and therir norms only.

\begin{lemma}\label{lem:fi-vs-convex}
Let $X,Y$ be normed spaces, and $\vec f\in X^n,\vec g\in Y^n$.

Let $\mathcal E_K$ be the John ellipsoid of $K:=\cave{\vec f}_{X}$ such that
\begin{equation*}
  \mathcal E_K\subset K\subset\sqrt{n}\mathcal E_K,
\end{equation*}
and suppose that $\mathcal E_K$ is non-degenerate (i.e., of full dimension).
Let $R_K$ be a linear transformation such that $R_K\mathcal E_K=\bar B_{\R^n}$, the closed unit ball of $\R^n$, and let $(\vec e_i)_{i=1}^n$ be an orthonormal basis of $\R^n$. If
\begin{equation*}
  f_i:=R_K\vec f\cdot \vec e_i,\quad g_i:=R_K^{-t}\vec g\cdot e_i,\qquad i=1,\ldots,n,
\end{equation*}
then
\begin{equation*}
  \sum_{i=1}^n\Norm{f_i}{X}\Norm{g_i}{Y}\leq n^{3/2}\cave{\vec f}_{X}\cdot\cave{\vec g}_{Y}.
\end{equation*}
\end{lemma}

\begin{proof}
If $\phi\in\bar B_{X^*}$, then
\begin{equation*}
\begin{split}
  \pair{\phi}{R_K\vec f\cdot\vec e_i}
  &=R_K\pair{\phi}{\vec f}\cdot\vec e_i \\
  &\in R_K\cave{\vec f}_{X}\cdot\vec e_i
  \subset R_K \sqrt{n}\mathcal E_K\cdot\vec e_i
  =\sqrt{n}\bar B_{\R^n}\cdot\vec e_i=\sqrt{n}[-1,1],
\end{split}
\end{equation*}
and hence
\begin{equation*}
  \Norm{f_i}{X}=\Norm{R_K\vec f\cdot\vec e_i}{X}\leq \sqrt{n}.
\end{equation*}

If $\psi\in\bar B_{Y^*}$, then
\begin{equation*}
\begin{split}
  \pair{\psi}{R_K^{-t}\vec g\cdot\vec e_i}
  &=R_K^{-t}\pair{\psi}{\vec g}\cdot\vec e_i \\
  &\in R_K^{-t} \cave{\vec g}_{Y}\cdot\vec e_i
  \subset [-M,M],\quad M:=\max\{\abs{\vec y}:\vec y\in R_K^{-t} \cave{\vec g}_{Y}\}.
\end{split}
\end{equation*}
It follows that $\abs{\pair{\psi}{R_K^{-t}\vec g\cdot\vec e_i}}\leq M$, and hence
\begin{equation*}
  \Norm{g_i}{Y}=\Norm{R_K^{-t}\vec g\cdot\vec e_i}{Y}\leq M.
\end{equation*}

Combining the estimates, we have
\begin{equation*}
  \sum_{i=1}^n\Norm{f_i}{X}\Norm{g_i}{Y}
  \leq\sum_{i=1}^n \sqrt{n}M=n^{3/2}M.
\end{equation*}

On the other hand,
\begin{equation*}
  \cave{\vec f}_{X}\cdot\cave{\vec g}_{Y}
  \supset \mathcal E_K\cdot\cave{\vec g}_{Y} 
  =R_K\mathcal E_K\cdot R_K^{-t}\cave{\vec g}_{Y} 
  =\bar B_{\R^n}\cdot R_K^{-t}\cave{\vec g}_{Y} 
  =[-M,M],
\end{equation*}
and hence
\begin{equation*}
  M  \leq \cave{\vec f}_{X}\cdot\cave{\vec g}_{Y},
\end{equation*}
using the identification of the symmetric interval $\cave{\vec f}_{X}\cdot\cave{\vec g}_{Y}$ with its right end-point in the last step.
Substituting back, this completes the proof.
\end{proof}

The following proposition contains the basic idea of bootstrapping norm bounds to convex body bounds. Unfortunately, it is a bit too simple for most actual applications, but we include it as an illustrative toy model for the more serious result to be presented after it.

\begin{proposition}\label{prop:loc-sparse2convex-pre}
Let $X,Y$ be normed spaces with subspaces $F\subset X$ and $G\subset Y$, and let $t:F\times G\to\R$ be a bilinear form.
Consider the following conditions:
\begin{enumerate}
  \item\label{it:1scaleDomSca-pre} For all $(f,g)\in F\times G$, we have
\begin{equation*}
  \abs{t(f,g)}\leq C\Norm{f}{X}\Norm{g}{Y}.
\end{equation*}
  \item\label{it:1scaleDomVec-pre} For all $(\vec f,\vec g)\in F^n\times G^n $, we have
\begin{equation*}
  \abs{t(\vec f,\vec g)}\leq C_n\cave{\vec f}_{X}\cdot\cave{\vec g}_{Y}.
\end{equation*}
\end{enumerate}
For each $n\in\Z_+$, condition \eqref{it:1scaleDomSca} implies condition \eqref{it:1scaleDomVec} with $C_n=Cn^{3/2}$.
\end{proposition}

\begin{proof}
Given $\vec f$, consider the compact, convex, symmetric set 
\begin{equation*}
  K:=\cave{\vec f}_{X}, 
\end{equation*}
and denote by $\mathcal E_K$ its John ellipsoid such that
\begin{equation*}
  \mathcal E_K\subset K\subset\sqrt{n}\mathcal E_K. 
\end{equation*}

\subsubsection*{Case: $\mathcal E_K$ is non-degenerate}
Let $R_K$ be a linear transformation such that $R_K\mathcal E_K=\bar B_{\R^n}$, the closed unit ball of $\R^n$.
Let $(\vec e_i)_{i=1}^n$ be some orthonormal basis of $\R^n$.
We then write
\begin{equation}\label{eq:tvecfvecg-pre}
\begin{split}
  t(\vec f,\vec g)
  &=t(R_K^{-1}R_K\vec f,\vec g)
  =t(R_K\vec f,R_K^{-t}\vec g) \\
  &=\sum_{i=1}^n t(R_K\vec f\cdot\vec e_i,R_K^{-t}\vec g\cdot\vec e_i)
    =:\sum_{i=1}^n t(f_i,g_i),
\end{split}
\end{equation}
where $f_i$ and $g_i$ are as in Lemma \ref{lem:fi-vs-convex}.

By assumption \eqref{it:1scaleDomSca-pre} and Lemma \ref{lem:fi-vs-convex}, it follows that
\begin{equation*}
  \abs{ t(\vec f,\vec g)}
    \leq\sum_{i=1}^n \abs{t(f_i,g_i)}
    \leq \sum_{i=1}^n C\Norm{f_i}{X}\Norm{g_i}{Y}
    \leq Cn^{3/2}\cave{\vec f}_X\cdot\cave{\vec g}{Y},
\end{equation*}
and this completes the proof in the case that $\mathcal E_K$ is non-degenerate.

\subsubsection*{Case: $\mathcal E_K$ is degenerate}
Suppose then that $\mathcal E_K$ is degenerate; hence $H:=\lspan K$ is a strict subspace of $\R^n$. Let $P$ denote the orthogonal projection of $\R^n$ onto $H$. For each $x^*\in\bar B_{X^*}$, we have $\pair{\vec f}{x^*}\in K\subset H$, hence
\begin{equation*}
  \pair{\vec f}{x^*}=P\pair{\vec f}{x^*}=\pair{P\vec f}{x^*},
\end{equation*}
and thus $\vec f=P\vec f$. It follows that 
\begin{equation}\label{eq:tfg-vs-tfPg}
  t(\vec f,\vec g)=  t(P\vec f,\vec g)=  t(\vec f,P^t\vec g)=  t(\vec f,P\vec g),
\end{equation}
and similarly
\begin{equation}\label{eq:caveg-vs-cavePg}
\begin{split}
  \cave{\vec f}_{X}\cdot\cave{\vec g}_{Y}
  =P\cave{\vec f}_{X}\cdot\cave{\vec g}_{Y} 
  =\cave{\vec f}_{X}\cdot P^t\cave{\vec g}_{Y}
 =\cave{\vec f}_{X}\cdot \cave{P\vec g}_{Y}.
\end{split}
\end{equation}
So it is enough to prove the claim with $P\vec g$ in place of $\vec g$, and hence we may assume without loss of generality that also $\vec g=P\vec g$. But then we can repeat the argument in the non-degenerate case, but with $\R^n$ replaced by its subspace $H$ throughout; within this subspace, $\mathcal E_K\subset H$ is non-degenerate, and the previous case applies to give the desired result.
\end{proof}

In the following proposition, condition \eqref{it:1scaleDomSca} is a typical intermediate step that is established in the course of proving a sparse domination result for an operator, while condition \eqref{it:1scaleDomVec} is its convex body analogue. The proposition says that  \eqref{it:1scaleDomSca} in fact implies \eqref{it:1scaleDomVec}. It is essentially an abstraction (from $L^1$ averages to general dominating norms) of an idea already present in \cite[Lemma 3.2]{NPTV:convex}.

\begin{proposition}\label{prop:loc-sparse2convex}
Let $X,Y$ be normed spaces with subspaces $F\subset X$ and $G\subset Y$.
Let $Q_0\in\mathscr D$, and suppose that there are bilinear forms $t_Q:F\times G\to\R$ indexed by all $Q\in\mathscr D(Q_0)$.
Consider the following conditions:
\begin{enumerate}
  \item\label{it:1scaleDomSca} For all $(f,g)\in F\times G$, there exist disjoint $\hat Q_k\subset Q_0$ with $\sum_k\abs{\hat Q_k}\leq\eps\abs{Q_0}$ and such that: whenever $Q_j\subset Q_0$ are disjoint, not strictly contained in any $\hat Q_k$, and cover all $\hat Q_k$, then
\begin{equation*}
  \abs{t_{Q_0}(f,g)-\sum_j t_{Q_j}(f,g)}\leq C\Norm{f}{X}\Norm{g}{Y}.
\end{equation*}
  \item\label{it:1scaleDomVec} For all $(\vec f,\vec g)\in F^n\times G^n $, there exist disjoint $Q_k\subset Q_0$ with $\sum_k\abs{Q_k}\leq\eps_n\abs{Q_0}$ and such that
\begin{equation*}
  \abs{t_{Q_0}(\vec f,\vec g)-\sum_k t_{Q_k}(\vec f,\vec g)}\leq C_n\cave{\vec f}_{X}\cdot\cave{\vec g}_{Y}.
\end{equation*}
\end{enumerate}
For each $n\in\Z_+$, condition \eqref{it:1scaleDomSca} implies condition \eqref{it:1scaleDomVec} with $\eps_n=n\eps$ and $C_n=Cn^{3/2}$.
\end{proposition}

Of course, the condition $\sum_k\abs{Q_k}\leq\eps_n\abs{Q_0}$ is only useful for $\eps_n<1$. For a fixed $\eps$ and $\eps_n=n\eps$, this would only allow us to conclude \eqref{it:1scaleDomVec} for boundedly many values of $n$; so in order to obtain \eqref{it:1scaleDomVec} for all $n\in\N$, we need \eqref{it:1scaleDomSca} for arbitrarily small $\eps>0$. This is seldom a problem in concrete situations.

\begin{proof}
As in the proof of Proposition \ref{prop:loc-sparse2convex-pre}, given $\vec f$, we consider the compact, convex, symmetric set 
\begin{equation*}
  K:=\cave{\vec f}_{X}, 
\end{equation*}
and denote by $\mathcal E_K$ its John ellipsoid such that
\begin{equation*}
  \mathcal E_K\subset K\subset\sqrt{n}\mathcal E_K. 
\end{equation*}

\subsubsection*{Case: $\mathcal E_K$ is non-degenerate}
Let $R_K$ be a linear transformation such that $R_K\mathcal E_K=\bar B_{\R^n}$, the closed unit ball of $\R^n$.
Let $(\vec e_i)_{i=1}^n$ be some orthonormal basis of $\R^n$.
As in \eqref{eq:tvecfvecg-pre}, we then write
\begin{equation}\label{eq:tvecfvecg}
\begin{split}
  t_{Q_0}(\vec f,\vec g)
  &=t_{Q_0}(R_K^{-1}R_K\vec f,\vec g)
  =t_{Q_0}(R_K\vec f,R_K^{-t}\vec g) \\
  &=\sum_{i=1}^n t_{Q_0}(R_K\vec f\cdot\vec e_i,R_K^{-t}\vec g\cdot\vec e_i)
    =:\sum_{i=1}^n t_{Q_0}(f_i,g_i),
\end{split}
\end{equation}
where $f_i$ and $g_i$ are as in Lemma \ref{lem:fi-vs-convex}.

It is from this point on that the present proof requires some elaboration compared to the proof of Proposition \ref{prop:loc-sparse2convex-pre}. According to assumption \eqref{it:1scaleDomSca}, for each of the pairs of functions $f_i:=R_K\vec f\cdot\vec e_i$ and $g_i:=R_K^{-t}\vec g\cdot\vec e_i$, we can find disjoint $\hat Q_{i,k}\subset Q_0$ with $\sum_k\abs{\hat Q_{i,k}}\leq\eps\abs{Q_0}$ and such that: whenever $Q_j\subset Q_0$ are disjoint, not strictly contained in any $\hat Q_{i,k}$, and cover all $\hat Q_{i,k}$, then
\begin{equation}\label{eq:tfigi}
  \abs{t_{Q_0}(f_i,g_i)-\sum_j t_{Q_j}(f_i,g_i)}\leq C\Norm{f_i}{X}\Norm{g_i}{Y}.
\end{equation}
We make the following specific choice of the cubes $Q_j$: Let $\{Q_j\}_{j=1}^\infty$ be the maximal cubes among $\{\hat Q_{i,k}\}_{1\leq i\leq n}^{1\leq k<\infty}$. Then
\begin{equation*}
  \sum_j\abs{Q_j}\leq\sum_{i=1}^n\sum_{k=1}^\infty\abs{\hat Q_{i,k}}
  \leq\sum_{k=1}^n\eps\abs{Q_0}
  = n\eps\abs{Q_0},
\end{equation*}
and \eqref{eq:tfigi} holds with these $Q_j$ for each $i=1,\ldots,n$. Using \eqref{eq:tvecfvecg}, and observing that it also holds with $Q_0$ replaced by $Q_j$, it follows that
\begin{equation*}
\begin{split}
   &\abs{t_{Q_0}(\vec f, \vec g)-\sum_j t_{Q_j}(\vec f,\vec g)}
   \leq\sum_{i=1}^n \abs{t_{Q_0}(f_i,g_i)-\sum_j t_{Q_j}(f_i,g_i)} \\
   &\qquad\leq C\sum_{i=1}^n\Norm{f_i}{X}\Norm{g_i}{Y}
   \leq Cn^{3/2}\cave{\vec f}_X\cdot\cave{\vec g}_Y,
\end{split}
\end{equation*}
using Lemma \ref{lem:fi-vs-convex} in the last step.
This completes the proof under the assumption that $\mathcal E_K$ is non-degenerate.

\subsubsection*{Case: $\mathcal E_K$ is degenerate}
This follows the corresponding case in the proof of Proposition \ref{prop:loc-sparse2convex-pre} almost verbatim. Like there, let $H:=\lspan K$, and let $P$ denote the orthogonal projection of $\R^n$ onto $H$. We then have \eqref{eq:tfg-vs-tfPg} for each $t=t_Q$, as well as \eqref{eq:caveg-vs-cavePg}. So it is again enough to prove the claim with $P\vec g$ in place of $\vec g$, and hence we may assume without loss of generality that also $\vec g=P\vec g$. But then we can repeat the argument in the non-degenerate case, but with $\R^n$ replaced by its subspace $H$ throughout; within this subspace, $\mathcal E_K\subset H$ is non-degenerate, and the previous case applies to give the desired result.
\end{proof}


\section{From single-scale bounds to global bounds}

This passage is by now a relatively routine part of the theory, but we include some details for completeness. The following lemma is again stated in an operator-free, and even function-free way, simply as a criterion for dominating a real number by sum over a sparse collection. A more concrete situation for applying this criterion is presented afterwards.

\begin{lemma}\label{lem:loc2glob}
Consider numbers $a\in\R$ and $a_Q,c_Q\in\R$ indexed by dyadic cubes $Q\in\mathscr D$, with the following properties:
\begin{enumerate}
  \item\label{it:cover} There is a family $\mathscr Q$ of disjoint dyadic cubes such that
\begin{equation*}
   a=\sum_{Q\in\mathscr Q}a_Q.
\end{equation*}
  \item\label{it:local} For some $\delta\in(0,1)$ and each $Q\in\mathscr D$ that is contained in some $P\in\mathscr Q$, there is a family of disjoint $Q_k\in\mathscr D(Q)$ such that
\begin{equation*}
   \sum_k\abs{Q_k}\leq\delta\abs{Q},\qquad
   \Babs{a_{Q}-\sum_k a_{Q_k}}\leq c_{Q}.
\end{equation*}
  \item\label{it:limit} For some $\alpha,C\in[1,\infty)$ and each $Q\in\mathscr D$ that is contained in some $P\in\mathscr Q$, we have $\abs{a_Q}\leq C\abs{Q}^\alpha$.
\end{enumerate}
Then there is a $(1-\delta)$-sparse family of dyadic cubes $\mathscr S$ such that
\begin{equation*}
   \mathscr S\subset\bigcup_{Q\in\mathscr Q}\mathscr D(Q),\qquad  \abs{a}\leq \sum_{S\in\mathscr S}c_S.
\end{equation*}
\end{lemma}

\begin{remark}\label{rem:loc2glob}
If $\mathscr Q=\{Q_0\}$ consists of a single cube only, then condition \eqref{it:cover} is automatic with $a=a_{Q_0}$.
\end{remark}

\begin{proof}
Let $\mathscr Q\subset\mathscr D$ be a disjoint collection provided by assumption \eqref{it:cover}. For each $P\in\mathscr Q$, denote $\mathscr S_0(P):=\{P\}$.
 Assuming that a disjoint $\mathscr S_j(P)\subset\mathscr D(P)$ has already been constructed, for each $Q\in\mathscr S_j(P)$, let $\mathscr S'(Q):=\{Q_k\}_{k=1}^\infty$ be the collection provided by assumption \eqref{it:local}, and let $\mathscr S_{j+1}(P):=\bigcup_{Q\in\mathscr S_j(P)}\mathscr S'(Q)$. Let also 
  $\mathscr S(P):=\bigcup_{j=0}^\infty\mathscr S_j(P)$, and $\mathscr S:=\bigcup_{P\in\mathscr Q}\mathscr S(P)$.

For $Q\in\mathscr S$, let $E(Q):=Q\setminus\bigcup_{R\in\mathscr S'(Q)}R$. From the construction it is clear that these sets $E(Q)$ are pairwise disjoint, and by assumption \eqref{it:local} we have $\abs{E(Q)}\geq(1-\delta)\abs{Q}$.

By telescoping, for each $P\in\mathscr Q$, we have
\begin{equation*}
  a_P=\sum_{j=0}^{k-1}\sum_{Q\in\mathscr S_j(P)}\Big(a_Q-\sum_{R\in\mathscr S'(Q)}a_R\Big)+\sum_{S\in\mathscr S_k(P)}a_S.
\end{equation*}
and hence, using assumptions \eqref{it:local} and \eqref{it:limit},
\begin{equation*}
  \abs{a_P}\leq\sum_{j=1}^{k-1}\sum_{Q\in\mathscr S_j(P)}c_Q+\sum_{S\in\mathscr S_k(P)} C\abs{S}^\alpha
\end{equation*}
By an elementary inequality and induction, we have
\begin{equation*}
  \sum_{S\in\mathscr S_k(P)} \abs{S}^\alpha
  \leq\Big(\sum_{S\in\mathscr S_k(P)} \abs{S}\Big)^\alpha\leq(\delta^k\abs{P})^\alpha,
\end{equation*}
and hence
\begin{equation*}
  \abs{a_P}\leq\lim_{k\to\infty}\sum_{j=1}^{k-1}\sum_{Q\in\mathscr S_j(P)}c_Q=\sum_{Q\in\mathscr S(P)}c_Q.
\end{equation*}
Substituting this into assumption \eqref{it:cover}, we obtain the claim.
\end{proof}

\begin{lemma}\label{lem:easyConds}
Suppose that $t$ is a bilinear form on $L^\infty_c(\R^d;E)\times L^\infty_c(\R^d;H)$, and moreover bounded with respect to the norm of $L^p(\R^d;E)\times L^q(\R^d;H)$ for some exponents with $1/p+1/q\geq 1$. For $(\vec f,\vec g)\in L^\infty_c(\R^d;E)^n\times L^\infty_c(\R^d;H)^n$, the numbers
\begin{equation*}
   a=t(\vec f,\vec g),\quad a_Q=t(1_{3Q}\vec f,1_Q \vec g)
\end{equation*}
satisfy assumptions \eqref{it:cover} and \eqref{it:limit} of Lemma \ref{lem:loc2glob}, provided that $\mathscr D$ is a dyadic system without quadrants.
\end{lemma}

\begin{proof}
Since $\mathscr D$ is without quadrants, each $Q\in\mathscr D$ is contained in some (large enough) $R\in\mathscr D$ that contains $\supp \vec f$. Thus the collection $\mathscr Q$ of maximal cubes that do not contain $\supp \vec f$ form a cover of $\R^d$. By maximality, it follows that $\supp \vec f\subset 3Q$, and hence $\vec f=1_{3Q}\vec f$ for every $Q\in\mathscr Q$. On the other hand, any $Q$ with $\ell(Q)<\operatorname{diam}(\supp \vec f)$ cannot contain $\supp\vec f$; hence any $Q$  with $\ell(Q)<\frac12\operatorname{diam}(\supp \vec f)$ cannot be among the maximal cubes $\mathscr Q$, and thus every $Q\in\mathscr Q$ will have to satisfy $\ell(Q)\geq\frac12\operatorname{diam}\vec f$. Since $\vec g\in L^\infty_c(\R^d;F)^n$, there are only finitely many $Q\in\mathscr Q$ with $1_Q \vec g\neq 0$. Hence, without any issues of convergence, we can write
\begin{equation*}
  t(\vec f,\vec g)=t\Big(\vec f,\sum_{Q\in\mathscr Q}1_Q \vec g\Big)=\sum_{Q\in\mathscr Q}t(\vec f,1_Q \vec g)=\sum_{Q\in\mathscr Q}t(1_{3Q}\vec f,1_Q \vec g),
\end{equation*}
which is condition \eqref{it:cover}.

If $n=1$, the assumed boundedness directly implies that
\begin{equation*}
   \abs{t(1_{3Q}f,1_Q g)}\leq C\Norm{1_{3Q}f}{L^p(\R^d;E)}\Norm{1_Q g}{L^{q}(\R^d;F)}
   \leq C 3^{d/p}\Norm{f}{\infty}\Norm{g}{\infty}\abs{Q}^{1/p+1/q},
\end{equation*}
where $\alpha:=1/p+1/q\geq 1$, as required for condition \eqref{it:limit}. In general, if $(\vec e_i)_{i=1}^n$ is an orthonormal basis of $\R^n$ and $\vec f=\sum_{i=1}^n f_i\vec e_i$ and similarly for $\vec g$, we have
\begin{equation*}
  \abs{t(1_{3Q}\vec f,1_Q \vec g)}\leq\sum_{i=1}^n\abs{t(1_{3Q}f_i,1_Q g_i)}
   \leq Cn 3^{d/p}\Norm{\vec f}{\infty}\Norm{\vec g}{\infty}\abs{Q}^{1/p+1/q},
\end{equation*}
using the previous bound in each component and trivial bounds like $\Norm{f_i}{\infty}\leq\Norm{\vec f}{\infty}$.
\end{proof}

We are finally ready to state a semi-generic convex body domination principle. Condition \eqref{it:1scaleDomScaNew} below is a typical intermediate estimate in a number of sparse domination proofs for different operators. The conclusion is that it is already good enough to conclude convex body domination as well.

\begin{corollary}\label{cor:1scale2cbd}
Let $E$ and $H$ be Banach spaces, and suppose that $t$ is a bilinear form defined on $F\times G:=L^\infty_c(\R^d;E)\times L^\infty_c(\R^d;H)$ and bounded with respect to the norm of $L^p(\R^d;E)\times L^q(\R^d;H)$ for some exponents with $1/p+1/q\geq 1$, and suppose that 
\begin{enumerate}
  \item\label{it:1scaleDomScaNew} for all $(f,g)\in F\times G$ and all $Q\in\mathscr D$, there are disjoint $\hat Q_k\subset Q$ with $\sum_k\abs{\hat Q_k}\leq\eps\abs{Q}$ and such that: whenever $Q_j\subset Q$ are disjoint, not strictly contained in any $\hat Q_k$, and cover all $\hat Q_k$, then
\begin{equation}\label{eq:1scaleDomScaNew}
  \abs{t(1_{3Q}f,1_{Q}g)-\sum_j t(1_{3Q_j}f,1_{Q_j}g)}\leq c\Norm{f}{X(Q)}\Norm{g}{Y(Q)}\abs{Q}
\end{equation}
for some norms $\Norm{\ }{X(Q)}$ on $L^\infty_c(\R^d;E)$ and $\Norm{\ }{Y(Q)}$ on $L^\infty_c(\R^d;H)$.
\end{enumerate}
Then for all $(\vec f,\vec g)\in F^n\times G^n$, there is a $(1-\eps_n)$-sparse collection $\mathscr S\subset\mathscr D$ such that
\begin{equation*}
   \abs{t(\vec f,\vec g)}\leq c_n\sum_{S\in\mathscr S}\cave{\vec f}_{X(S)}\cdot\cave{\vec g}_{Y(S)}\abs{S},
\end{equation*}
where $\eps_n=n\eps$ and $c_n=cn^{3/2}$.
\end{corollary}

\begin{proof}
Let us begin by considering a fixed cube $Q=Q_0\in\mathscr D$. We observe that assumption \eqref{it:1scaleDomScaNew} of the present corollary coincides with condition \eqref{it:1scaleDomSca} of  Proposition \ref{prop:loc-sparse2convex} with
\begin{equation*}
   t_Q(f,g):=t(1_{3Q}f,1_{Q}g),\qquad C=c\abs{Q},\qquad X=X(Q),\qquad Y=Y(Q).
\end{equation*}
Hence the said proposition, applied to each fixed $Q=Q_0\in\mathscr D$ at a time, implies:
\begin{enumerate}\setcounter{enumi}{1}
  \item\label{it:1scaleDomScaVecNew} For all $(\vec f,\vec g)\in L^\infty_c(\R^d;E)^n\times L^\infty_c(\R^d;H)^n$ and all $Q\in\mathscr D$, there are disjoint $Q_k\subset Q$ with $\sum_k\abs{Q_k}\leq\eps_n\abs{Q}$ and such that
\begin{equation*}
  \abs{t(1_{3Q}\vec f,1_{Q}\vec g)-\sum_j t(1_{3Q_j}\vec f,1_{Q_j}\vec g)}\leq c_n\cave{\vec f}_{X(Q)}\cdot \cave{g}_{Y(Q)}\abs{Q},
\end{equation*}
where $\eps_n=n\eps$ and $c_n=cn^{3/2}$.
\end{enumerate}

Let us then consider a fixed pair $(\vec f,\vec g)\in L^\infty_c(\R^d;E)^n\times L^\infty_c(\R^d;H)^n$. We observe that condition \eqref{it:1scaleDomScaVecNew} above coincides with condition \eqref{it:local} of Lemma \ref{lem:loc2glob} with the choices
\begin{equation*}
    a_Q=t(1_{3Q}\vec f,1_{Q}\vec g),\qquad c_Q= c_n\cave{\vec f}_{X(Q)}\cdot \cave{g}_{Y(Q)}\abs{Q},\qquad \delta=\eps_n.
\end{equation*}
On the other hand, Lemma \ref{lem:easyConds} shows that these same $a_Q$, together with $a:=t(\vec f,\vec g)$, also satisfy conditions \eqref{it:cover} and \eqref{it:limit} of Lemma \ref{lem:loc2glob}. Thus, all assumptions, and hence the conclusions, of Lemma \ref{lem:loc2glob} are valid for the said quantities, and these conclusions agree with the claims of the result that we are proving. The proof is thus complete.
\end{proof}

To facilitate the discussion of consequences of Corollary \ref{cor:1scale2cbd}, we give

\begin{definition}\label{def:cbd}
Suppose that a pair of normed spaces $(X(Q),Y(Q))$ is associated to every dyadic cube $Q\in\mathscr D$.
We say that a bilinear form $t:F\times G\to\R$ satisfies the $(X(Q),Y(Q))$ convex body domination of order $n\in\N$ if $F\subseteq X(Q)$ and $G\subseteq Y(Q)$ for every $Q\in\mathscr D$, and if for every $(f,g)\in F^n\times G^n$, there exists a $\delta_n$-sparse collection $\mathscr S\subset\mathscr D$ such that
\begin{equation*}
    \abs{t(\vec f,\vec g)}\leq C_n\sum_{Q\in\mathscr S}\abs{Q}\cave{\vec f}_{X(Q)}\cdot\cave{\vec g}_{Y(Q)}.
\end{equation*}
We say that $t:F\times G\to\R$ satisfies the $(X(Q),Y(Q))$ convex body domination if it satisfies this for every $n\in\N$. We say that an operator $T:F\to G^*$ satisfies these properties if its associated bilinear form $t(f,g):=\pair{Tf}{g}$ does.
\end{definition}

Let us now consider some examples:

\begin{example}[Calder\'on--Zygmund operators]\label{ex:CZO}
Let $T$ be a Dini--Calder\'on--Zygmund operator, i.e., $T$ is $L^2(\R^d)$ bounded and has the representation
\begin{equation*}
  Tf(x)=\int_{\R^d}K(x,y)f(y)\ud y,\qquad x\notin\supp f,
\end{equation*}
where $\abs{K(x,y)}\leq c\abs{x-y}^{-d}$ and, for $\abs{x-x'}\leq\frac12\abs{x-y}$,
\begin{equation}\label{eq:CZomega}
  \abs{K(x,y)-K(x',y)}+\abs{K(y,x)-K(y,x')}\leq\omega\Big(\frac{\abs{x-x'}}{\abs{x-y}}\Big)\frac{1}{\abs{x-y}^d},
\end{equation}
where $\omega:[0,\frac12]\to[0,\infty)$ is increasing, subadditive, and satisfies the Dini condition
\begin{equation*}
  \int_0^{1/2}\omega(t)\frac{\ud t}{t}<\infty.
\end{equation*}
Then \eqref{it:1scaleDomScaNew} of Corollary \ref{cor:1scale2cbd} holds for $t(f,g)=\pair{Tf}{g}$ and $E=H=\R$ and $X(Q)=\avL^1(3Q)$, $Y(Q)=\avL^1(Q)$, even in a stronger form. Namely, on the left oif \eqref{eq:1scaleDomScaNew}, we have
\begin{equation}\label{eq:Lerner}
\begin{split}
     &\Babs{\Bpair{T(1_{3Q}f)}{1_{Q}g}-\sum_j \pair{T(1_{3Q_j}f)}{1_{Q_j}g}} \\
     &\leq\BNorm{1_QT(1_{3Q}f)-\sum_j 1_{Q_j}T(1_{3Q_j}f)}{L^\infty(Q)}\Norm{g}{L^1(Q)},
\end{split}
\end{equation}
and even the $L^\infty$ norm here is dominated by $\Norm{f}{\avL^1(3Q)}$, as essentially shown in \cite[(3.4)]{Lerner:NYJM}. (Strictly speaking, \cite[(3.4)]{Lerner:NYJM} is formally slightly weaker, but a straightforward modification of the argument gives the desired version, as observed in \cite[Proof of Theorem 3.4]{NPTV:convex}.) Thus Corollary \ref{cor:1scale2cbd} says that a Dini--Calder\'on--Zygmund operator satisfies $(\avL^1(3Q),\avL^1(Q))$ convex body domination, but this was of course already known from \cite{NPTV:convex} by essentially the same argument.
\end{example}

\begin{example}[Banach space -valued Calder\'on--Zygmund operators]\label{ex:CZO-Banach}
Let $T$ be as in Example \ref{ex:CZO} but now acting on the Bochner space $L^2(\R^d;E)$ of Banach space $E$ -valued functions, and with an operator-valued kernel $K(x,y)\in\bddlin(E)$ satisfying the same estimates as above but for the operator norm in place of the absolute value, e.g., $\Norm{K(x,y)}{\bddlin(E)}\leq c\abs{x-y}^{-d}$. It is in general a difficult problem to check the $L^2(\R^d;E)$-boundedness of such an operator, but we now take this as an assumption. For $g\in L^2(\R^d;E^*)$, we have \eqref{eq:Lerner} with $L^\infty(Q;E)$ and $L^1(Q;E^*)$ in place of $L^\infty(Q)$ and $L^1(Q)$, and the same proof of \cite[(3.4)]{Lerner:NYJM} (with same modifications pointed out in \cite[Proof of Theorem 3.4]{NPTV:convex}) shows that the $L^\infty(Q;E)$ norm is dominated by $\Norm{f}{\avL^1(3Q;E)}$. Thus we find that \eqref{it:1scaleDomScaNew} of Corollary \ref{cor:1scale2cbd} also holds with $X(Q)=\avL^1(3Q;E)$ and $Y(Q)=\avL^1(Q;E^*)$. The resulting sparse domination (i.e., case $n=1$ of the conclusion of Corollary \ref{cor:1scale2cbd}) was known before, first in \cite{HH:14} for a slightly smaller class of kernels, and since \cite[discussion on page 193]{Lacey:elem} in the present generality. However, the convex body domination in this Banach space -valued setting is completely new.
\end{example}

\begin{example}[Operators with grand maximal function control]\label{ex:grand}
Let $1 \leq q \leq r$ and $s \geq  1$. Suppose that $T$ is a linear operator
\begin{equation}\label{eq:TLcL1loc}
  T:L^\infty_c(\R^d)\to L^1_{\loc}(\R^d),
\end{equation}
that $T$ has weak type $(q,q)$, and that the bi-sublinear maximal operator
\begin{equation*}
  \mathcal M_T(f,g)(x):=\sup_{Q\owns x} \fint_Q \abs{T(1_{(3Q)^c}f)}\cdot\abs{g}
\end{equation*}
maps boundedly $\mathcal M_T:L^r\times L^s\to L^{\nu,\infty}$, where $1/\nu=1/r+1/s$. Then condition \eqref{it:1scaleDomScaNew} of Corollary \ref{cor:1scale2cbd} holds for $t(f,g)=\pair{Tf}{g}$ and $E=H=\R$ and $X(Q)=\avL^r(3Q)$, $Y(Q)=\avL^s(Q)$. This result is essentially contained in the proof of \cite[Theorem 3.1]{Lerner:rough}, where it appears as an intermediate step towards the sparse domination (i.e., case $n=1$ of the conclusion of Corollary \ref{cor:1scale2cbd}) for such operators. The extension to convex body domination was recently achieved in \cite{MRR:22}, so Corollary \ref{cor:1scale2cbd} only reproduces a known result here. A key example of concrete operators satisfying these assumptions consists of rough homogeneous singular integrals
\begin{equation*}
  Tf(x)=\int_{\R^d}\frac{\Omega(y)}{\abs{y}^d}f(x-y)\ud y,
\end{equation*}
where $\Omega(y)=\Omega(y/\abs{y})$ is a bounded function with vanishing average over the unit sphere.

As in Example \ref{ex:CZO-Banach}, the abstract result above, involving a priori bounds of $T$ and $\mathcal M_T$, extends straightforwardly to the Banach space -valued setting; however, verifying these bounds for concrete operators such as the rough homogeneous singular integrals may present a problem in this generality, since the scalar-valued versions depend on deep results of Seeger \cite{Seeger:rough}, which so far lack a Banach space -valued extension.
\end{example}


\section{Matrix-weighted inequalities for Banach space -valued operators}\label{sec:L2WE}

A matrix weight is a locally integrable function $W:\R^d\to\R^{n\times n}$ that is a.e.\ positive definite -valued. The space $L^p(W)$ consists of all measurable $\vec f:\R^d\to\R^n$ such that $W^{1/p}\vec f\in L^p(\R^d;\R^n)$, and $\Norm{\vec f}{L^p(W)}:=\Norm{W^{1/p}\vec f}{L^p(\R^d;\R^n)}$.

For a Banach space $E$, we extend this definition in a natural way: The space $L^p(W;E^n)$ consists of all measurable $\vec f:\R^d\to E^n$ such that $W^{1/p}\vec f\in L^p(\R^d;E^n)$, and $\Norm{\vec f}{L^p(W;E^n)}:=\Norm{W^{1/p}\vec f}{L^p(\R^d;E^n)}$. Here, at each $x\in\R^d$, we define $(W^{1/p}\vec f)(x)\in E^n$ as the vector with components $(W^{1/p}\vec f)_i(x):=\sum_{j=1}^n (W^{1/p}(x))_{ij}f_j(x)$, i.e., the matrix multiplication on $\R^n$ is extended to $E^n$ in the natural way.

We now concentrate on $p=2$. For two matrix weights $W,V:\R^d\to\R^{n\times n}$, we define
\begin{equation*}
  [W,V]_{A_2}:=\sup_Q\abs{\ave{W}_Q^{1/2}\ave{V}_Q^{1/2}}^2,\qquad
  [W]_{A_2}:=[W,W^{-1}]_{A_2},
\end{equation*}
where we denote the operator norm in $\R^{n\times n}\simeq\bddlin(\R^n)$ simply by $\abs{\ }$. We denote by $A_2(\R^d;\R^n)$ the class of matrix weights $W:\R^d\to\R^{n\times n}$ for which $[W]_{A_2}<\infty$. We also define
\begin{equation*}
  [W]_{A_\infty}:=\sup_{\vec e\in\R^n}[x\mapsto \vec e\cdot W(x)\vec e]_{A_\infty},
\end{equation*}
where on the right we have $A_\infty$ ``norms'' of some scalar weights, defined as usual by
\begin{equation*}
  [w]_{A_\infty}:=\sup_Q\frac{1}{w(Q)}\int_Q M(1_Q w).
\end{equation*}
According to \cite[Remark 4.4]{NPTV:convex}, we have
\begin{equation}\label{eq:Ainfty-vs-A2}
  [W]_{A_\infty}\leq 4[W]_{A_2}.
\end{equation}

As a consequence of the Banach space -valued convex body domination from Example \ref{ex:CZO-Banach}, we obtain:

\begin{theorem}\label{thm:L2WE}
Let $E$ be a Banach space, and $T\in\bddlin(L^2(\R^d;E))$ be a Dini--Calder\'on--Zygmund operator with $\bddlin(E)$-valued kernel. For any $W\in A_2(\R^d;\R^n)$, the operator $T$ extends boundedly to $L^2(W;E^n)$ and satisfies
\begin{equation*}
  \Norm{T}{\bddlin(L^2(W;E^n))}\leq c_{n,T}([W]_{A_2}[W]_{A_\infty}[W^{-1}]_{A_\infty})^{1/2}\leq c_{n,T}[W]_{A_2}^{3/2}.
\end{equation*}
\end{theorem}

Note that Theorem \ref{thm:L2WE} applies to a general Banach space $E$, but contains the (difficult) a priori boundedness hypothesis that $T\in\bddlin(L^2(\R^d;E))$. Concrete examples are available in the class of UMD spaces, treated in detail in \cite{HNVW}.

\begin{corollary}\label{cor:L2WE}
Let $E$ be a UMD space, and $T\in\bddlin(L^2(\R^d))$ be a scalar-valued Calder\'on--Zygmund operator with a H\"older-type modulus of continuity $\omega(t)=c t^\delta$, $\delta\in(0,1]$ in \eqref{eq:CZomega}. For any $W\in A_2(\R^d;\R^n)$, the operator $T$ extends boundedly to $L^2(W;E^n)$ and satisfies
\begin{equation*}
  \Norm{T}{\bddlin(L^2(W;E^n))}\leq c_{n,E,T}([W]_{A_2}[W]_{A_\infty}[W^{-1}]_{A_\infty})^{1/2}\leq c_{n,E,T}[W]_{A_2}^{3/2}.
\end{equation*}
In particular, this estimate holds when $T$ is the classical Hilbert transform.
\end{corollary}

\begin{proof}
We reduce Corollary \ref{cor:L2WE} to Theorem \ref{thm:L2WE} with the help of the $T(1)$ theorem of David and Journ\'e \cite{DJ:T1}, and its extension to UMD spaces by Figiel \cite{Figiel:T1}. By the (easy half of) the David--Journ\'e theorem, the assumptions on $T$ imply that that $T$ satisfies the so-called weak boundedness property as well as $T(1),T^*(1)\in\BMO(\R^d)$. Then, by Figiel's theorem, an operator satisfying these conditions and the Calder\'on--Zygmund kernel assumptions extends boundedly to $L^2(\R^d;E)$, for any UMD space $E$. Thus $T$ satisfies the assumptions, and hence the conclusions, of Theorem \ref{thm:L2WE}, and we are done.
\end{proof}

These results, even just for the Hilbert transform, and even in their qualitative form (i.e., just concluding the boundedness of $T$, without specifying any concrete bound for the norm), are completely new in the combined setting of matrix weights and Banach spaces. For $E=\R$ and the Hilbert transform $T$, the qualitative form of Corollary \ref{cor:L2WE} is due to Treil and Volberg \cite{TV:angle}. The quantitative form for $E=\R$ was obtained by Nazarov et al.\ \cite{NPTV:convex}, and this is the best that is known at the time of writing. For scalar-weights, the power $3/2$ can be replaced by $1$ \cite{Hytonen:A2}, and the product of $[W]_{A_\infty}$ and $[W^{-1}]_{A_\infty}$ by their sum \cite{HytPer}, but extending these to the general matrix case consists of the outstanding open ``matrix $A_2$ conjecture''.

Turning to the proof of Theorem \ref{thm:L2WE}, we begin with:

\begin{remark}[Without loss of generality, we assume that $E$ is reflexive]\label{rem:RNP}
Since Theorem \ref{thm:L2WE} is about the bounded extension of an operator, it suffices to prove an a priori estimate on a dense subspace of functions $\vec f$. In particular, we can assume that each component $f_i$ takes its values in a finite-dimensional subspace of $E$. Since any finite-dimensional space is reflexive, {\em we make the standing assumption, without loss of generality, that $E$ is reflexive}. (Note that this is automatic in Corollary \ref{cor:L2WE} in any case, since UMD spaces are reflexive \cite[Theorem 4.3.3]{HNVW}.) Under this assumption, we have $L^1(Q;E)^*=L^\infty(Q;E^*)$ (see \cite[Theorems 1.3.10 and 1.3.21]{HNVW}), which is convenient in view of calculations involving the convex bodies $\cave{\ }_{\avL^1(Q;E)}$.
\end{remark}

\begin{lemma}
\begin{equation*}
\begin{split}
  \abs{Q} &\cave{W^{1/2}\vec f}_{\avL^1(3Q;E)}\cdot \cave{V^{1/2}\vec g}_{\avL^1(Q;E^*)} \\
  &\leq \int \Big(1_Q(x)\fint_{3Q}\abs{V^{1/2}(x)W^{1/2}(y)}\Norm{\vec f(y)}{E^n}\ud y\Big)\Norm{\vec g(x)}{\vec E^{*n}}\ud x
\end{split}
\end{equation*}
\end{lemma}

\begin{proof}
Under the standing assumption from Remark \ref{rem:RNP}, we evaluate consider a generic element of the convex body on the left with $\phi\in\bar B_{L^\infty(Q;E^*)}$ and $\psi\in\bar B_{L^\infty(Q;E)}$:
\begin{equation*}
\begin{split}
  &\abs{Q}\Babs{\fint_{3Q} W^{1/2}(y)\pair{\vec f(y)}{\phi(y)}\ud y\cdot \fint_Q V^{1/2}(x)\pair{\vec g(x)}{\psi(x)}\ud x} \\
  &=\abs{Q}\Babs{\fint_Q\fint_{3Q} V^{1/2}(x)W^{1/2}(y)\pair{\vec f(y)}{\phi(y)}\cdot\pair{\vec g(x)}{\psi(x)}\ud y\ud x} \\
  &\leq\int_Q\fint_{3Q} \abs{V^{1/2}(x)W^{1/2}(y)}\Norm{\vec f(y)}{E^n}\Norm{\vec g(x)}{\vec E^{*n}}\ud y\ud x.\qedhere
\end{split}
\end{equation*}
\end{proof}

Summing over a sparse collection, we obtain
\begin{equation}\label{eq:cbd-vs-L}
\begin{split}
   &\sum_{Q\in\mathscr S} \abs{Q} \cave{W^{1/2}\vec f}_{\avL^1(3Q;E)}\cdot \cave{V^{1/2}\vec g}_{\avL^1(Q;E^*)}  \\
   &\leq \int\Big(\sum_{Q\in\mathscr S}1_Q(x)\fint_{3Q}\abs{V^{1/2}(x)W^{1/2}(y)}\Norm{\vec f(y)}{E^n}\ud y\Big)\Norm{\vec g(x)}{\vec E^{*n}}\ud x \\
   &=:\int \tilde L(\Norm{\vec f}{E^n})(x)\Norm{\vec g(x)}{\vec E^{*n}}\ud x,
\end{split}
\end{equation}
where $\tilde L$, here acting on the scalar-valued function $y\mapsto\Norm{\vec f(y)}{E^n}$, is an operator denoted by the same symbol in \cite[(5.8)]{NPTV:convex}. By \cite[Lemma 5.6]{NPTV:convex}, we have
\begin{equation}\label{eq:NPTV-L}
  \Norm{\tilde L}{\bddlin(L^2)}\leq C([W,V]_{A_2}[W]_{A_\infty}[V]_{A_\infty})^{1/2}.
\end{equation}

By duality and standard changes of variables, which present no essential difference in the Banach space -valued setting, an estimate of the form
\begin{equation*}
  \Norm{T\vec f}{L^2(V;E^n)}\leq N\Norm{\vec f}{L^2(V;E^n)}
\end{equation*}
is equivalent to
\begin{equation}\label{eq:dualBound2prove}
  \pair{T(W^{1/2}\vec f)}{V^{1/2}\vec g}\leq N\Norm{\vec f}{L^2(\R^d;E^n)}\Norm{\vec g}{L^2(\R^d;E^{*n})}.
\end{equation}
If $T$ is an in Theorem \ref{thm:L2WE}, it satisfies the $(\avL^1(3Q;E),\avL^1(Q;E^*))$ convex body domination by Example \ref{ex:CZO-Banach}, which means that the left-hand side of \eqref{eq:dualBound2prove} is dominated by the left-hand side of \eqref{eq:cbd-vs-L}, and hence, by \eqref{eq:cbd-vs-L} and \eqref{eq:NPTV-L}, we have
\begin{equation*}
  N\leq c_{n,T}([W,V]_{A_2}[W]_{A_\infty}[V]_{A_\infty})^{1/2}.
\end{equation*}
This is the desired bound, and concludes the proof of Theorem \ref{thm:L2WE}.

\section{Convex domination and generalised commutators}\label{sec:commu}

For an operator $T$ and two vector functions $\vec a=(a_1,\ldots,a_n)$ and $\vec b=(b_1,\ldots,b_n)$, let us consider the operator
\begin{equation*}
  \vec a\cdot T\vec b:f\mapsto\vec a\cdot T(\vec b f)=\sum_{i=1}^n a_i T(b_i f).
\end{equation*}
We are mainly interested in the boundedness on $L^p(\R^d)$, or a weighted $L^p(w)$, or between two such spaces, and the case when $T$ is a singular integral operator bounded on the space. However, we do not require that $a_i,b_i\in L^\infty(\R^d)$, and hence the pointwise multipliers $f\mapsto b_i f$ and $g\mapsto a_i g$, and the compositions $f\mapsto a_iT(b_i f)$, may be unbounded operators. Nevertheless, their sum $\vec a\cdot T\vec b$ may still be bounded, thanks to cancellation between different terms.

A case that has been much studied in the literature consists of $\vec b=(1,b)$ and $\vec a=(b,-1)$, in which case
\begin{equation*}
  \vec a\cdot T(\vec b f)= bTf-T(bf)=[b,T]f
\end{equation*}
is the {\em commutator} of $b$ and $T$, whose $L^p(\R^d)$-boundedness is characterised by $b\in\BMO(\R^d)$, the space of functions of bounded mean oscillation, which is strictly larger than $L^\infty(\R^d)$, and contains in particular functions like $b(x)=\log\abs{x}$.

By dualising with a function $g$, and denoting by $t(f,g)=\pair{Tf}{g}$ the bilinear form of $T$, we arrive at
\begin{equation*}
  \pair{\vec a\cdot T(\vec b f)}{g}=\sum_{i=1}^n\pair{T(b_i f)}{a_i g}
  =t(\vec b f,\vec a g),
\end{equation*}
where the action of the bilinear form is extended to vector-valued functions as before. To be precise, if $t$ in defined on $F\times G$, we should now require that
\begin{equation*}
  f\in F_{\vec b}:=\{f\in F: b_i f\in F\text{ for all }i=1,\ldots,n\},
\end{equation*}
and $g\in G_{\vec a}$, defined similarly. If $F\supseteq L^\infty_c(\R^d)$, then clearly $F_{\vec b}$ contains in particular all $f\in L^\infty_c(\R^d)$ with $\supp f\subseteq E_N:=\{\abs{\vec b}\leq N\}$ for any $N\in\N$. For a.e.\ finite-valued $b_i$, the union $\bigcup_{N\in\N}E_N$ covers $\R^d$ up to a null set, it is immediate that $F_{\vec b}$ is dense in any $L^p(w)$ with finite $p$.

\begin{lemma}\label{lem:cbd-implies}
Suppose that $T$ satisfies the $(X(Q),Y(Q))$ convex body domination. Then for all relevant functions, we have
\begin{equation}\label{eq:cbd-implies}
  \abs{\pair{\vec{a}\cdot T(\vec{b}f)}{g}}
  \leq C\sum_{Q\in\mathscr S}\abs{Q}\cave{\vec b f}_{X(Q)}\cdot\cave{\vec a g}_{Y(Q)}.
\end{equation}
\end{lemma}

\begin{proof}
This is immediate by applying definition to $\vec f=\vec b f$ and $\vec g=\vec a g$.
\end{proof}

We take a closer look at the case when $X(Q)=Y(Q)=\avL^1(\gamma Q)$.

\begin{lemma}\label{lem:cbdL1-further}
For all $s,t\in(1,\infty)$ and all functions in the relevant spaces, we have
\begin{equation*}
  \cave{\vec b f}_{\avL^1(Q)}\cdot\cave{\vec a g}_{\avL^1(Q)}
  \leq\Norm{(x,y)\mapsto\vec a(x)\cdot\vec b(y)}{\avL^{(s,t)}_{\min}(Q\times Q)}
    \Norm{f}{\avL^{t'}(Q)}\Norm{g}{\avL^{s'}(Q)},
\end{equation*}
where
\begin{equation*}
  \Norm{F}{\avL^{(s,t)}_{\min}(Q\times Q)}
  :=\begin{cases} \Big(\fint_Q\Big[\fint_Q\abs{F(x,y)}^s\ud x\Big]^{t/s}\ud y\Big)^{1/t}, & \text{if }s\leq t, \\
   \Big(\fint_Q\Big[\fint_Q\abs{F(x,y)}^t\ud y\Big]^{s/t}\ud x\Big)^{1/s}, & \text{if }t\leq s. \end{cases} 
\end{equation*}
\end{lemma}

\begin{proof}
The generic element of $\cave{\vec b f}_{X(Q)}\cdot\cave{\vec a g}_{Y(Q)}$ has the following form, where $\phi,\psi\in\bar B_{L^\infty(Q)}$:
\begin{equation*}
\begin{split}
  &\fint_Q \vec b(y)f(y)\phi(y)\ud y\cdot\fint_Q \vec a(x)g(x)\psi(x)\ud x \\
  &=\fint_Q\fint_Q(\vec a(x)\cdot\vec b(y))f(y)g(x)\phi(y)\psi(x)\ud x\ud y,
\end{split}
\end{equation*}
and hence
\begin{equation*}
\begin{split}
  &\cave{\vec b f}_{X(Q)}\cdot\cave{\vec a g}_{Y(Q)}
  \leq\fint_Q\fint_Q\abs{\vec a(x)\cdot\vec b(y)} \abs{f(y)}\abs{g(x)}\ud x\ud y \\
  &\leq\Norm{(x,y)\mapsto a(x)\cdot b(y)}{Z}\Norm{(x,y)\mapsto f(y)g(x)}{Z^*},
\end{split}
\end{equation*}
for either choice of
\begin{equation*}
  (Z,Z^*)\in\{(\avL^s_x(Q;\avL^t_y(Q)),\avL^{s'}_x(Q;\avL^{t'}_y(Q))),
  (\avL^t_y(Q;\avL^s_x(Q)),\avL^{t'}_y(Q;\avL^{s'}_x(Q)))\},
\end{equation*}
by H\"older's inequality for mixed-norm $L^p$ spaces. By Fubini's theorem, we have
\begin{equation*}
   \Norm{(x,y)\mapsto f(x)g(y)}{Z^*}
   =\Norm{f}{\avL^{t'}(Q)}\Norm{g}{\avL^{s'}(Q)}
\end{equation*}
in either case, and hence, taking the minimum over the two choices of $Z$, we arrive at the factor
\begin{equation*}
  \min_Z\Norm{(x,y)\mapsto b(x)\cdot a(y)}{Z}=\Norm{(x,y)\mapsto\vec b(x)\cdot\vec a(y)}{\avL^{(s,t)}_{\min}(Q\times Q)}.
\end{equation*}
\end{proof}

\begin{proposition}\label{prop:gen-commu}
Let $T$ be an operator that satisfies the $(\avL^1(\gamma Q),\avL^1(\gamma Q))$ convex body domination. Let $\vec a,\vec b\in L^1_{\loc}(\R^d)^n$ be functions such that
\begin{equation*}
  A_{s,t}:=\sup_{Q}\Norm{(x,y)\mapsto \vec a(x)\cdot\vec b(y)}{\avL^{(s,t)}_{\min}(Q\times Q)}<\infty.
\end{equation*}
Then $\vec a\cdot T\vec b$ extends to a bounded operator on $L^p(\R^d)$ for all $p\in(t',s)$. In particular, if $A_s:=A_{s,s}<\infty$ for some $s\in(2,\infty)$, then $\vec a\cdot T\vec b$ extends boundedly to $L^2(\R^d)$.
\end{proposition}

\begin{proof}
Combining Lemmas \ref{lem:cbd-implies} and \ref{lem:cbdL1-further}, we find that
\begin{equation*}
\begin{split}
  \abs{\pair{\vec aT(\vec b f)}{g}}
  &\leq C\sum_{Q\in\mathscr S}\abs{Q}\cave{\vec b f}_{\avL^1(\gamma Q)}\cdot\cave{\vec a g}_{\avL^1(\gamma Q)} \\
  &\leq C\sum_{Q\in\mathscr S}\abs{Q}\Norm{(x,y)\mapsto a(x)\cdot b(y)}{\avL^{(s,t)}_{\min}(Q\times Q)}
     \Norm{f}{\avL^{t'}(Q)}\Norm{g}{\avL^{s'}(Q)} \\
   &\leq C\sum_{Q\in\mathscr S}\frac{\abs{E(Q)}}{\delta}A_{s,t}\inf_Q M_{t'}f\inf_Q M_{s'}g \\
   &\leq\frac{C A_{s,t}}{\delta}\sum_{Q\in\mathscr S}\int_{E(Q)}M_{t'}f M_{s'}g
   \leq\frac{C A_{s,t}}{\delta}\int_{\R^d}M_{t'}f M_{s'}g \\
   &\leq\frac{C A_{s,t}}{\delta}\Norm{M_{t'}f}{L^p(\R^d)}\Norm{M_{s'}g}{L^{p'}(\R^d)},
\end{split}
\end{equation*}
where
\begin{equation*}
  \Norm{M_{t'}f}{L^p(\R^d)}\lesssim_{t,p}\Norm{f}{L^p(\R^d)},\qquad
  \Norm{M_{s'}g}{L^{p'}(\R^d)}\lesssim_{s,p}\Norm{g}{L^{p'}(\R^d)}
\end{equation*}
for $p>t'$ and $p'>s'$, where the latter is equivalent to $p<s$.
\end{proof}

Let us consider some examples:

\begin{example}[Classical commutators]\label{ex:commu} 
As we already observed, $\vec a=(b,-1)$ and $\vec b=(1,b)$ gives rise to the usual commutator $[b,T]$. In this case
\begin{equation*}
  \vec a(x)\cdot\vec b(y)= b(x)-b(y)
\end{equation*}
and each $A_{s,t}$ is equivalent to $\Norm{b}{\BMO(\R^d)}$ by elementary considerations and the John--Nirenberg inequality. Thus Proposition \ref{prop:gen-commu} reproduces the well-known sufficient condition for the boundedness of commutators.
\end{example}

\begin{example}[Iterated commutators]\label{ex:iter-commu} More generally, choosing $\vec a,\vec b$ so that
\begin{equation*}
   \vec a(x)\cdot\vec b(y)= (b(x)-b(y))^k=\sum_{i=0}^k\binom{k}{i}b(x)^{k-i}(-b(y))^i,
\end{equation*}
thus e.g. $a_i(x)=\binom{k}{i}b(x)^{k-i}$ and $b_i(y)=(-b(y))^i$, we reproduce the $k$th order commutator
\begin{equation*}
  \vec a\cdot T\vec b=T_{k,b}:=[b,T_{k-1,b}],\qquad T_{0,b}:=T,
\end{equation*}
and $A_{s,t}$ is equivalent to $\Norm{b}{\BMO(\R^d)}^k$ by the John--Nirenberg inequality.
\end{example}

\begin{example}[Iterated commutators with different multipliers]\label{ex:mixed-commu}
Let us then choose $\vec a,\vec b$ so that
\begin{equation*}
   \vec a(x)\cdot\vec b(y)= (b^1(x)-b^1(y))(b^2(x)-b^2(y));
\end{equation*}
without specifying the precise choice of $a_i(x)$ and $b_i(y)$, it is evident that such a choice can be easily written down, if desired. (We deliberately use superscript indices for $b^i$ above, since these not be the same as the components $b_i$ of $\vec b$.) This reproduces the second order iterated commutator with two different functions,
\begin{equation*}
   \vec a\cdot T\vec b=[b^1,[b^2,T]].
\end{equation*}
It is well-known and classical that $b^i\in\BMO(\R^d)$ for both $i=1,2$ is sufficient for the $L^2(\R^d)$ boundedness of $[b^1,[b^2,T]]$; however, as recently observed in \cite{HLO:20}, much weaker sufficient conditions can be given for the pair $(b^1,b^2)$. Namely, in \cite[(1.1)]{HLO:20}, it shown that the pair of conditions
\begin{equation*}
\begin{split}
  S_s &:=\sup_Q\Big(\fint_Q\abs{b^1(x)-\ave{b^1}_Q}^s\ud x\Big)^{1/s}\Big(\fint_Q\abs{b^2(y)-\ave{b^2}_Q}^s\ud y\Big)^{1/s}<\infty, \\
  T_s &:=\sup_Q\Big(\fint_Q\abs{b^1(x)-\ave{b^1}_Q}^s\abs{b^2(x)-\ave{b^2}_Q}^s\ud x\Big)^{1/s}<\infty,
\end{split}
\end{equation*}
is sufficient for the $L^2(\R^d)$ boundedness of $[b^1,[b^2,T]]$ for $s>2$. On the other hand, by Proposition \ref{prop:gen-commu}, another sufficient condition for the same conclusion is $A_s<\infty$.

Let us compare the two. Adding and subtracting terms and multiplying out, we find that
\begin{equation*}
\begin{split}
  &(b^1(x)-b^1(y))(b^2(x)-b^2(y)) \\
  &=[(b^1(x)-\ave{b^1}_Q)-(b^1(y)-\ave{b^1}_Q)][(b^2(x)-\ave{b^2}_Q)-(b^2(y)-\ave{b^2}_Q)] \\
  &=(b^1(x)-\ave{b^1}_Q)(b^2(x)-\ave{b^2}_Q)+(b^1(y)-\ave{b^1}_Q)(b^2(y)-\ave{b^2}_Q) \\
  &\qquad-(b^1(x)-\ave{b^1}_Q)(b^2(y)-\ave{b^2}_Q)-(b^1(y)-\ave{b^1}_Q)(b^2(x)-\ave{b^2}_Q).
\end{split}
\end{equation*}
Taking $\avL^s(Q\times Q)$ and then supremum over $Q$ on both sides, we deduce that
\begin{equation*}
  A_s\leq 2(T_s+S_s),
\end{equation*}
so that the new criterion provided by Proposition \ref{prop:gen-commu} is at least as sharp as that of \cite[(1.1)]{HLO:20}, and it seems less obvious to make any estimate in the other direction. Perhaps more importantly, the new condition $A_s<\infty$ arises more ``naturally'' as an instance of a general principle.

(Let us note that there is a more general criterion \cite[Theorem 3.10]{HLO:20}, where the $\avL^s$ norms in $S_s$ and $T_t$ are replaced by more general Orlicz norms. On the other hand, it is apparent that similar generalisations could be achieved in Proposition \ref{prop:gen-commu}: what we used was the boundedness of the rescaled maximal operators $M_{t'}$ on $L^p(\R^d)$ for $p>t'$, and this could be replaced having an Orlicz maximal operator $M_A$ with the same mapping property. A characterisation of this property in terms of the so-called $B_p$ condition on the Orlicz function $A$ is a classical result of P\'erez \cite{Perez:95}; this very result is used in \cite{HLO:20}; see \cite[Proposition 3.8]{HLO:20}.)
\end{example}

Let us finally consider an ``exotic'' example with no obvious predecessor in the existing literature. We begin with a lemma:

\begin{lemma}\label{lem:BMOpowers}
Suppose that $0\leq b\in\BMO(\R^d)$. If $0\leq\alpha,\beta$ and $\alpha+\beta\leq 1$, then
\begin{equation*}
  B(x,y):=b(x)^\alpha b(y)^\beta - b(x)^\beta b(y)^\alpha
\end{equation*}
satisfies
\begin{equation*}
  \Big(\fint_Q\fint_Q\abs{B(x,y)}^p\ud x\ud y\Big)^{1/p} \leq (2\Norm{b}{\BMO^p(\R^d)})^{\alpha+\beta}.
\end{equation*}
\end{lemma}

\begin{proof}
Let $\gamma:=\min(\alpha,\beta)\in[0,\frac12]$ and $\delta:=\max(\alpha,\beta)-\gamma\in[0,1]$. Then
\begin{equation*}
  \abs{B(x,y)}=b(x)^\gamma b(y)^\gamma \abs{b(x)^\delta-b(y)^\delta}.
\end{equation*}
We observe the following elementary inequality:
\begin{equation}\label{eq:ud-vd}
  \abs{u^\delta-v^\delta}\leq\frac{\abs{u-v}}{\max(u,v)^{1-\delta}},\qquad\forall u,v\geq 0,\ \delta\in[0,1].
\end{equation}
Indeed, by symmetry and homogeneity, it is enough to consider $u=1$ and $v\in[0,1]$, in which case we are reduced to proving that
\begin{equation*}
  1-v^\delta\leq 1-v,
\end{equation*}
which is immediate from the fact that $v\leq v^\delta$ for $v,\delta\in[0,1]$.

Using \eqref{eq:ud-vd}, and noting that $\delta+2\gamma=\alpha+\beta\in[0,1]$, it follows that
\begin{equation*}
\begin{split}
  \abs{B(x,y)} &\leq b(x)^\gamma b(y)^\gamma\frac{\abs{b(x)-b(y)}}{\max(b(x),b(y))^{1-\delta}}  
  \leq \frac{\abs{b(x)-b(y)}}{\max(b(x),b(y))^{1-\delta-2\gamma}} \\
  &= \Big(\frac{\abs{b(x)-b(y)}}{\max(b(x),b(y))}\Big)^{1-\delta-2\gamma}\abs{b(x)-b(y)}^{\delta+2\gamma}
  \leq\abs{b(x)-b(y)}^{\alpha+\beta},
\end{split}
\end{equation*}
and hence
\begin{equation*}
\begin{split}
  \Big( &\fint_Q\fint_Q\abs{B(x,y)}^p\ud x\ud y\Big)^{1/p}
  \leq \Big(\fint_Q\fint_Q\abs{b(x)-b(y)}^p\ud x\ud y\Big)^{(\alpha+\beta)/p} \\
  &\leq\Big[\Big(\fint_Q\abs{b(x)-c}^p\ud x\Big)^{1/p}+\Big(\fint_Q\abs{b(y)-c}^p\ud y\Big)^{1/p}\Big]^{\alpha+\beta}
\end{split}
\end{equation*}
for all constants $c$.
\end{proof}

\begin{corollary}
Let $T$ be an operator satisfying $(\avL^1(\gamma Q),\avL^1(\gamma Q))$ convex body domination, let $0\leq b\in\BMO(\R^d)$ and $0\leq\alpha,\beta$ with $\alpha+\beta\leq 1$. Then
\begin{equation*}
  \Norm{b^\alpha T(b^\beta f)-b^\beta T(b^\alpha f)}{L^p(\R^d)}\lesssim_p \Norm{b}{\BMO(\R^d)}^{\alpha+\beta}\Norm{f}{L^p(\R^d)}.
\end{equation*}
\end{corollary}

\begin{proof}
By Proposition \ref{prop:gen-commu} with $s=t$, the $L^p(\R^d)$ operator norm of $f\mapsto b^\alpha T(b^\beta f)-b^\beta T(b^\alpha f)$ is dominated by
\begin{equation*}
   A_s:= \sup_Q\Norm{(x.y)\mapsto b(x)^\alpha b(y)^\beta-b(x)^\beta b(y)^\alpha}{\avL^s(Q\times Q)}
\end{equation*}
if $p\in(s',s)$, i.e., if $s>\max(p,p')$. By Lemma \ref{lem:BMOpowers} and the John--Nirenberg inequality, we have
\begin{equation*}
   A_s\leq(2\Norm{b}{\BMO^s(\R^d)})^{\alpha+\beta}\lesssim_s\Norm{b}{\BMO(\R^d)}^{\alpha+\beta},
\end{equation*}
and fixing (say) $s=2\max(p,p')$, we obtain a dependence on $p$ only.
\end{proof}

\begin{remark}
Aside from the examples already discussed, the generalised commutators $\vec a\cdot T\vec b$ also arise in the following question studied by Bloom \cite{Bloom:85,Bloom:89}. Suppose that a matrix weight $W$ is given in the diagonalised form $W = U^*\Lambda U$, where $U$ is unitary, $\Lambda$ is diagonal, and the diagonal entries $\lambda_k$ of $\Lambda$ are scalar $A_2$ weights. What does one need to know about $U$ in order to conclude that $W\in A_2$? (According to \cite[Theorem 4.2]{Bloom:89}, the condition that $\lambda_k\in A_2$ is necessary for $W\in A_2$, if in addition $U$ is assumed to be continuous.)

Let $T$ be the Hilbert transform, or another operator whose boundedness on the matrix-weighted $L^2(W)$ characterises $W\in A_2$. By connecting the $L^2(W)$ boundedness of $T$ to the boundedness of the classical commutators $[T,\bar u_{ij}]$ between the weighted spaces $L^2(\lambda_i)$ and $L^2(\lambda_k)$ (sic: the condition involves triplets of indices $(i,j,k)$), \cite[Theorem 5.1]{Bloom:85} shows that $u_{ij}\in\BMO_{\sqrt{\lambda_i/\lambda_k}}$ (a weighted BMO space, nowadays commonly referred to as Bloom-type BMO) is a sufficient condition. In the special case of $2\times 2$ matrices, it is also necessary by \cite[Theorem 4.3]{Bloom:89} but, over 30 years since these contributions, the general case seems to remain open. (The author is grateful to Amalia Culiuc for bringing this question to his attention \cite{Culiuc:22}.)

Here is a possible approach to the problem. As is well known, the $L^2(W)$ boundedness of $T$ is equivalent to the (unweighted) $L^2$ boundedness of
\begin{equation*}
  W^{1/2}TW^{-1/2}=U^*\Lambda^{1/2}UTU^*\Lambda^{-1/2}U.
\end{equation*}
Multiplication by $U$ and $U^*$ is isometric on $L^2$, and the $L^2$ boundedness of a matrix of operators is equivalent to the $L^2$ boundedness of each of the components
\begin{equation*}
  (\Lambda^{1/2}UTU^*\Lambda^{-1/2})_{ij}
  =\sum_{k=1}^n\lambda_i^{1/2}u_{ik}T\bar u_{jk}\lambda_j^{-1/2}
  =\lambda_i^{1/2}\vec u_i\cdot T\bar {\vec u}_j\lambda_j^{-1/2},
\end{equation*}
where $i,j=1,\ldots,n$ and $\vec u_i=(u_{ik})_{k=1}^n$. These are operators of the form $\vec{a}\cdot T\vec{b}$ that we have studied here and, up to this point, we kept an exact equivalence with the original question; the question then would be, whether we can give useful conditions on the boundedness of these operators. A further equivalent condition is of course the two-weight boundedness
\begin{equation*}
  \vec u_i\cdot T\bar {\vec u}_j:L^2(\lambda_j)\to L^2(\lambda_i),\qquad i,j=1,\ldots,n,
\end{equation*}
where the spaces are more complicated, but the multipliers are simply rows of the unitary matrix $U$.
\end{remark}

\begin{remark}
We have concentrated in this section on the application of convex body domination---an inherently vector-valued theory---to questions of generalised commutators acting on scalar-valued functions. We have made this choice for two reasons: to make the case that this vector-valued theory is useful even for such scalar-valued applications, and not to obscure the relatively simple basic philosophy behind too many technicalities of notation. This said, it is quite plain that the presented ideas can be immediately generalised to the case of vector-valued functions $\vec f$ and $\vec g$ (in place of scalar $f$ and $g$) and matrix-valued multipliers $A$ and $B$ (in place of the vectors $\vec a$ and $\vec b$). In the particular case of the classical-style commutator $[T,B]$ with a matrix-valued function, this idea has been developed in \cite{IPT:commu}. 
\end{remark}

\section{Stopping times and maximal functions involving convex bodies}

The aims of this final section are two-fold. Concretely, we establish a convex-body analogue of a result of Nieraeth \cite{Nie:19}, which shows that the estimation of sums over sparse collection that arise in the usual sparse domination is equivalent to the estimation of certain maximal functions. On the way of achieving this, we develop some convex-body versions of the typical stopping time arguments involving averages of scalar-valued functions; these might have some independent interest elsewhere.

We begin with an estimate of a sum of convex-body ``norms'' over disjoint subsets.

\begin{lemma}\label{lem:stoppingAux}
Let $p,q\in[1,\infty)$ and $\frac1r:=\frac1p+\frac1q$. Let $Q_i\in\mathscr D(Q_0)$ be disjoint cubes. Then
\begin{equation*}
  \sum_{i=1}^\infty\big(\cave{\vec f}_{L^p(Q_i)}\cdot\cave{\vec g}_{L^q(Q_i)}\big)^r
  \leq n^{\max(r,1)+r/2}\big(\cave{f}_{L^p(Q)}\cdot\cave{g}_{L^q(Q)}\big)^r.
\end{equation*}
\end{lemma}

Note that for $p,q\in[1,\infty)$, we have $\frac1r=\frac1p+\frac1q\leq 1+1=2$, and hence
\begin{equation*}
  n^{\max(r,1)+r/2}
  =\big(n^{\max(1,1/r)+1/2}\big)^r\leq \big(n^{5/2}\big)^r.
\end{equation*}

\begin{proof}
For orientation, let us begin with the proof in the case $n=1$, i.e., with $\Norm{\ }{}$ in place of $\cave{\ }$ throughout. By H\"older's inequality with $1=\frac{r}{p}+\frac{r}{q}$, we have
\begin{equation*}
\begin{split}
   \sum_{i=1}^\infty &\big(\Norm{f}{L^p(Q_i)}\Norm{g}{L^q(Q_i)}\big)^r
   =\sum_{i=1}^\infty\big(\Norm{f}{L^p(Q_i)}^p\big)^{r/p}\big(\Norm{g}{L^q(Q_i)}^q\big)^{r/q} \\
   &\leq\Big(\sum_{i=1}^\infty\Norm{f}{L^p(Q_i)}^p\Big)^{r/p}\Big(\sum_{i=1}^\infty\Norm{g}{L^q(Q_i)}^q\Big)^{r/q} \\
   &\leq\Big(\Norm{f}{L^p(Q_0)}^p\Big)^{r/p}\Big(\Norm{g}{L^q(Q_0)}^q\Big)^{r/q}
     =\Big(\Norm{f}{L^p(Q_0)}\Norm{g}{L^q(Q_0)}\Big)^{r}.
\end{split}
\end{equation*}

In the general case of the lemma, let
\begin{equation*}
  A_i:=\cave{\vec f}_{L^p(Q_i)}
  =\Big\{\int_{Q_i}\phi_i\vec f:\Norm{\phi_i}{L^{p'}(Q_i)}\leq 1\Big\},\quad B_i:=\cave{\vec g}_{L^q(Q_i)}.
\end{equation*}
Then we observe that
\begin{equation*}
\begin{split}
  \cave{\vec f}_{L^p(Q)}
  &=\Big\{\int_{Q}\phi\vec f:\Norm{\phi}{L^{p'}(Q)}\leq 1\Big\} \\
  &\supseteq\Big\{\sum_{i=1}^\infty a_i\int_{Q_i}\phi_i\vec f:\Norm{\phi_i}{L^{p'}(Q_i)}\leq 1, \Norm{(a_i)}{\ell^{p'}}\leq 1\Big\} \\
  &=\Big\{\sum_{i=1}^\infty a_i A_i: \Norm{(a_i)}{\ell^{p'}}\leq 1\Big\}=:\bigoplus_{\ell^p}A_i=:A,
\end{split}
\end{equation*}
and similarly
\begin{equation*}
  \cave{\vec g}_{L^q(Q)}\supseteq\bigoplus_{\ell^q}B_i=:B.
\end{equation*}
Hence, the lemma is reduced to proving that
\begin{equation*}
  \sum_{i=1}^\infty\big(A_i\cdot B_i\big)^r
  \leq n^{\max(r,1)+r/2}\big(A\cdot B\big)^r,\quad A:=\bigoplus_{\ell^p}A_i,\quad B:=\bigoplus_{\ell^q}B_i.
\end{equation*}

Let $\mathcal E_A$ be the John ellipsoid of $A$, and let $R_A\mathcal E_A=\bar B_{\R^n}$. Since $A_i\cdot B_i=R_A A_i\cdot R_A^{-t}B_i$, The claim above is equivalent to a version where each $A_i$ is replaced by $R_A A_i$ and each $B_i$ by $R_A^{-t}B_i$. Hence, without loss of generality, we assume that $\mathcal E_A=\bar B_{\R^n}$ to begin with, hence $\bar B_{\R^n}\subseteq A\subseteq \sqrt{n}\bar B_{\R^n}$. Thus
\begin{equation*}
  A\cdot B\supset \bar B_{\R^n}\cdot B=[-M,M],\quad\text{where}\quad M:=\max\{\abs{\vec b}:\vec b\in B\}.
\end{equation*}
On the other hand, if $(\vec e_j)_{j=1}^n$ is some orthonormal basis of $\R^n$, then
\begin{equation*}
\begin{split}
  A_i\cdot B_i &=\{\vec a\cdot \vec b:\vec a\in A_i,\vec b\in B_i\} \\
  &=\Big\{\sum_{j=1}^n(\vec a\cdot\vec e_j)(\vec b\cdot\vec e_j):\vec a\in A_i,\vec b\in B_i\} 
  \subseteq\sum_{j=1}^n (A_i\cdot \vec e_j)(B_i\cdot\vec e_j),
\end{split}
\end{equation*}
or, using the identification of $[-s,s]$ with $s$,
\begin{equation*}
  A_i\cdot B_i \leq \sum_{j=1}^n (A_i\cdot \vec e_j)(B_i\cdot\vec e_j).
\end{equation*}
Thus
\begin{equation*}
  (A_i\cdot B_i)^r \leq \sum_{j=1}^n \big((A_i\cdot \vec e_j)(B_i\cdot\vec e_j)\big)^r,\quad r\in(0,1],
\end{equation*}
and
\begin{equation*}
  \Big(\sum_{i=1}^\infty (A_i\cdot B_i)^r\Big)^{1/r}\leq\sum_{j=1}^n \Big(\sum_{i=1}^\infty \big((A_i\cdot \vec e_j)(B_i\cdot\vec e_j)\big)^r\Big)^{1/r},\quad r\in[1,\infty).
\end{equation*}
In the sum over $i$, we use H\"older's inequality as in the toy model in the beginning:
\begin{equation*}
\begin{split}
  \sum_{i=1}^\infty&\big((A_i\cdot \vec e_j)(B_i\cdot\vec e_j)\big)^r
  =\sum_{i=1}^\infty\big((A_i\cdot \vec e_j)^p\big)^{r/p}\big((B_i\cdot\vec e_j)^q\big)^{r/q} \\
  &\leq\Big(\sum_{i=1}^\infty(A_i\cdot \vec e_j)^p\Big)^{r/p}\Big(\sum_{i=1}^\infty(B_i\cdot\vec e_j)^q\Big)^{r/q} \\
  &=\sup\Big\{\Big(\sum_{i=1}^\infty a_i A_i\cdot\vec e_j\Big)^{1/r}\Big(\sum_{i=1}^\infty b_i B_i\cdot\vec e_j\Big)^{1/r}:
    \Norm{(a_i)}{\ell^{p'}}\leq 1,\Norm{(b_i)}{\ell^{q'}}\leq 1\Big\} \\
  &=(A\cdot \vec e_j)^r(B\cdot\vec e_j)^r
\end{split}
\end{equation*}
Here
\begin{equation*}
  A\cdot\vec e_j\subseteq \sqrt{n}\bar B_{\R^n}\cdot\vec e_j=[-\sqrt{n},\sqrt{n}],\quad A\cdot\vec e_j\leq \sqrt{n},
\end{equation*}
and clearly
\begin{equation*}
  B\cdot\vec e_j\leq M.
\end{equation*}
Altogether, writing $s:=\max(r,1)$, we have
\begin{equation*}
\begin{split}
  \Big(\sum_{i=1}^\infty(A_i\cdot B_i)^r\Big)^{1/s}
  &\leq \sum_{j=1}^n   \Big[\sum_{i=1}^\infty(A_i\cdot \vec e_j)^r(B_i\cdot\vec e_j)^r\Big]^{1/s} \\
  &\leq \sum_{j=1}^n  \Big[ (A\cdot \vec e_j)^r(B\cdot\vec e_j)^r\Big]^{1/s} 
  \leq n [ n^{r/2} M^r]^{1/s},
\end{split}
\end{equation*}
and hence
\begin{equation*}
   \sum_{i=1}^\infty(A_i\cdot B_i)^r\leq n^s n^{r/2} M^r= n^{\max(1,r)+r/2}(A\cdot B)^r,
\end{equation*}
which remained to be proved.
\end{proof}

The following lemma is a convex-body analogue of the basic principle underlying the simplest stopping time constructions: for a function on a cube $Q_0$, the total measure of the subcubes, where the average of a function is much bigger than on the whole $Q_0$, can be at most a fraction of the measure of $Q_0$.

\begin{lemma}\label{lem:convexStopping}
Let $A,p,q\in[1,\infty)$ and let $Q_i\in\mathscr D(Q_0)$ be disjoint cubes such that
\begin{equation*}
  \cave{\vec f}_{\avL^p(Q_i)}\cdot\cave{\vec g}_{\avL^q(Q_i)}\geq A\cave{\vec f}_{\avL^p(Q_0)}\cdot\cave{\vec g}_{\avL^q(Q_0)}.
\end{equation*}
Then
\begin{equation*}
  \sum_{i=1}^\infty\abs{Q_i}\leq \frac{n^{\max(r,1)+r/2}}{A^{r}}\abs{Q_0},\qquad \frac1r:=\frac1p+\frac1q.
\end{equation*}
\end{lemma}

\begin{proof}
Directly from the definition, it is easy to extend the basic identity $\Norm{f}{\avL^p(Q)}=\abs{Q}^{-1/p}\Norm{f}{L^p(Q)}$ to convex bodies as
\begin{equation}\label{eq:avLvsL}
  \cave{\vec f}_{\avL^p(Q)}=\abs{Q}^{-1/p}\cave{\vec f}_{L^p(Q)}.
\end{equation}
From this, the assumption of the lemma can be rewritten as
\begin{equation*}
  \abs{Q_i}^{-1/p-1/q} \cave{\vec f}_{L^p(Q_i)}\cdot\cave{\vec g}_{L^q(Q_i)}\geq A\abs{Q_0}^{-1/p-1/q}\cave{\vec f}_{\L^p(Q_0)}\cdot\cave{\vec g}_{\L^q(Q_0)},
\end{equation*}
or, rearranging,
\begin{equation*}
  \abs{Q_i}\leq \frac{A^{-r}\abs{Q_0}}{\big(\cave{\vec f}_{\L^p(Q_0)}\cdot\cave{\vec g}_{\L^q(Q_0)}\big)^r}
      \big(\cave{\vec f}_{L^p(Q_i)}\cdot\cave{\vec g}_{L^q(Q_i)}\big)^r.
\end{equation*}
Summing over $i$ and using Lemma \ref{lem:stoppingAux}, we obtain the claim.
\end{proof}

We now obtain the following proposition, which is a convex body analogue of a result of Nieraeth \cite[Prop. 2.7; especially Eq. (2.7) for $m=1$]{Nie:19}. It says that estimating the sums over sparse collections, like those that arise from convex body domination, is equivalent to estimating related bi-sublinear maximal operators. In \cite[Prop. 2.7]{Nie:19}, the result is formulated as a set of equivalent conditions for a tuple of weights. The formulation below has no reference to weights as such, but as soon as one starts asking questions about the boundedness of either side on spaces like $L^s(W)\times L^{s'}(W')$, the proposition guarantees that one can equally well study this boundedness for the other side of the equivalence.

\begin{proposition}
For all $\delta\in(0,1)$, all dimensions $d,n\geq 1$, exponents $p,q\in[1,\infty)$, and functions $\vec f\in L^p_{\loc}(\R^d)^n$, $\vec g\in L^q_{\loc}(\R^d)^n$, we have the two-sided estimate
\begin{equation*}
  \sup_{\mathscr S}\sum_{Q\in\mathscr S}\cave{\vec f}_{\avL^p(Q)}\cdot\cave{\vec g}_{\avL^q(Q)}\abs{Q}
  \eqsim\BNorm{\sup_{Q\in\mathscr D}1_Q\cave{\vec f}_{\avL^p(Q)}\cdot\cave{\vec g}_{\avL^q(Q)}}{L^1(\R^d)},
\end{equation*}
where the supremum is taken over all $\delta$-sparse collections of dyadic cubes in $\R^n$, and the implied constants depend only on $n,p,q$, and $\delta$.
\end{proposition}

\begin{proof}
With $\vec f\in L^p_{\loc}(\R^d)^n$ and $\vec g\in L^q_{\loc}(\R^d)^n$ fixed, let us denote
\begin{equation*}
  a_Q:=\cave{\vec f}_{\avL^p(Q)}\cdot\cave{\vec g}_{\avL^q(Q)}.
\end{equation*}
The estimate $\lesssim$ is immediate: From $\delta$-sparseness, we have $\abs{Q}\leq\delta^{-1}\abs{E(Q)}$ for some disjoint sets $E(Q)$, and hence
\begin{equation*}
   \sum_{Q\in\mathscr S}a_Q\abs{Q}
   \leq\frac{1}{\delta}\sum_{Q\in\mathscr S}a_Q\abs{E(Q)}
   =\frac{1}{\delta}\int_{\R^d}\sum_{Q\in\mathscr S}a_Q 1_{E(Q)}
   \leq\frac{1}{\delta}\int_{\R^d}\sup_{Q\in\mathscr D}a_Q 1_Q.
\end{equation*}

The estimate $\gtrsim$ needs a bit more. By monotone convergence, it is enough to consider $\mathscr D(Q_0)$ in place of $\mathscr D$. Let $\mathscr S_0:=\{Q_0\}$. For some $A>1$ to be chosen and $Q\in\mathscr D(Q_0)$, let $\mathscr S'(Q)$ consist of all maximal $Q'\in\mathscr D(Q)$ such that $a_{Q'}>A a_Q$. By maximality, the cubes $Q'\in\mathscr S'(Q)$ are disjoint. By Lemma \ref{lem:convexStopping}, we have
\begin{equation*}
   \sum_{Q'\in\mathscr S'(Q)} \abs{Q'}\leq\frac{n^{\max(1,r)+r/2}}{A^r}\abs{Q}\leq(1-\delta)\abs{Q},\qquad\frac1r:=\frac1p+\frac1q,
\end{equation*}
provided that $A$ is chosen large enough, depending on $n,p,q$, and $\delta$. Hence, defining inductively $\mathscr S_{j+1}:=\bigcup_{Q\in\mathscr S_j}\mathscr S'(Q)$ and $\mathscr S:=\bigcup_{j=0}^\infty\mathscr S_j$, we find that $\mathscr S$ is $\delta$-sparse. If $Q\in\mathscr D(Q_0)$ and $S\in\mathscr S$ is the minimal stopping cube that contains $Q$, then $a_Q\leq Aa_S$ by the way that the cubes $S\in\mathscr S$ were chosen, hence
\begin{equation*}
  \sup_{Q\in\mathscr D(Q_0)}1_Q a_Q\leq  \sup_{S\in\mathscr S}1_S Aa_S\leq A\sum_{S\in\mathscr S}1_S a_S,
\end{equation*}
and thus
\begin{equation*}
  \BNorm{\sup_{Q\in\mathscr D(Q_0)}1_Q a_Q}{L^1(\R^d)}\leq A\BNorm{\sum_{S\in\mathscr S}1_S a_S}{L^1(\R^d)}
  =A\sum_{S\in\mathscr S}a_S\abs{S}.\qedhere
\end{equation*}
\end{proof}


\end{document}